\documentclass[a4paper,11pt,oneside]{scrartcl}
\usepackage[utf8]{inputenc}
\usepackage[T1]{fontenc}
\usepackage{lmodern}
\usepackage[english]{babel}
\usepackage{amssymb}
\usepackage{amsmath}
\usepackage{amsthm}
\usepackage[dvipsnames,svgnames,table]{xcolor}
\usepackage[colorlinks=true,linkcolor=RoyalBlue,citecolor=PineGreen,urlcolor=RoyalBlue]{hyperref}
\usepackage[paper=a4paper,left=25mm,right=25mm,top=25mm,bottom=25mm]{geometry}
\usepackage{graphicx}

\usepackage{float}
\usepackage{makecell}
\usepackage{csquotes}
\usepackage{caption}
\usepackage{booktabs}
\usepackage{tikz}
\usetikzlibrary{calc}
\usetikzlibrary{decorations.pathreplacing, decorations.markings}
\usetikzlibrary{shapes.misc}
\usepackage{thmtools, thm-restate}
\usepackage[nameinlink]{cleveref}

\usepackage{algpseudocode}
\usepackage{algorithm}
\usepackage{algcompatible}

\definecolor{colR}{HTML}{CC6677}
\definecolor{colB}{HTML}{6699CC}
\colorlet{colG}{DarkSeaGreen}
\definecolor{colO}{HTML}{DDCC77}
\colorlet{colbg}{white}
\colorlet{colfg}{black}
\colorlet{colgraphv}{colfg!75!colbg}
\colorlet{colgraphe}{colfg!65!colbg}

\tikzstyle{vertex}=[circle, draw=colbg, fill=colgraphv, inner sep=0pt, minimum size=7pt]
\tikzstyle{edge}=[line width=1.5pt,colgraphe]
\tikzstyle{labelsty}=[font=\scriptsize]
\tikzstyle{crossvertex}=[cross out, draw,minimum size = 10pt, line width=1pt]
\tikzstyle{genericgraph}=[dashed,black!70!white]
\tikzstyle{rededge}=[edge, color=colR]

\theoremstyle{definition}
\newtheorem{definition}{Definition}

\newtheorem{theorem}[definition]{Theorem}
\newtheorem{proposition}[definition]{Proposition}
\newtheorem{lemma}[definition]{Lemma}
\newtheorem{r-lemma}{R-Lemma}
\newtheorem{si-lemma}{SI-Lemma}
\newtheorem{corollary}[definition]{Corollary}

\newtheorem*{theorem*}{Theorem}

\theoremstyle{remark}

\newcommand{\Z}{\mathbb{Z}}

\DeclareMathOperator{\mindeg}{mindeg}
\DeclareMathOperator{\Iso}{Iso}
\DeclareMathOperator{\iso}{iso}
\DeclareMathOperator{\sur}{sur}
\DeclareMathOperator{\pot}{pot}
\DeclareMathOperator{\gsdeg}{sdeg}
\DeclareMathOperator{\pccgsdeg}{psdeg}
\DeclareMathOperator{\pccpot}{pccpot}
\newcommand{\surp}[3]{\sur_{#1}(#2,#3)}

\newcommand{\nauty}{\textsc{Nauty}}
\newcommand{\geng}{\textsc{Geng}}
\newcommand{\pickg}{\textsc{Pickg}}
\newcommand{\pcc}{PCC}
\newcommand{\pccgs}{\pcc\ game state}

\newcommand{\gsix}[1]{\texttt{#1}}

\algrenewcommand{\algorithmiccomment}[1]{\hfill$\triangleright$ #1}
\newcommand{\lineref}[1]{\hyperref[#1]{Line~\ref{#1}}}

\newcommand{\unsure}{\texttt{Unsure}}

\pgfkeys{/pgf/number format/.cd,
int detect,
/pgf/number format/set thousands separator={\,}}

\Crefformat{equation}{#2(#1)#3}

\title{The minimum degree question for the Maker Breaker Domination Game}
\date{}
\usepackage{authblk}

\author[1]{Jakob Führer}
\author[2]{Georg Grasegger}
\author[3]{Paul Hametner}
\author[1]{Oliver Roche-Newton}
\affil[1]{Johannes Kepler University Linz, Institute for Algebra}
\affil[2]{Johann Radon Institute for Computational and Applied Mathematics, Austrian Academy of Sciences}
\affil[3]{Institute of Science and Technology, Austria}

\begin{document}
\maketitle
\begin{abstract}
    The Maker Breaker Domination Game is a two player game played on a graph $G$ in which the players take turns to claim a vertex from the graph. The aim of the Dominator is to claim the vertices of a dominating set, and the aim of the Staller is to prevent this. In this paper, we consider the following problem: for a given integer $d$, what is the size of the smallest (with respect to the number of vertices) graph with minimum degree $d$ such that the Dominator loses going first? We write $\beta(d)$ to denote the answer to this question. We determine the precise value of $\beta(d)$ for $d\leq 3$. For general $d$ it was known that $2^{d+1} \leq \beta(d) \leq 2^{d+1}+2d$; the upper bound is due to a construction communicated to us by Valentin Gledel, while the lower bound follows from a simple application of the Erd\H{o}s-Selfridge Theorem. We improve the lower bound to $\beta(d) \geq 2^{d+1}+2$.
\end{abstract}

\section{Introduction}\label{sec:introduction}

A Maker-Breaker game is a two-player positional game. The rules of the game are as follows. We are given a ground set $X$ and a family $\mathcal F \subset \mathcal P(X)$ of \textit{winning sets}. The players take turns to move, and each move consists of the player claiming an element of $X$. Once an element has been claimed by a player, it cannot be claimed again. The goal of the Maker is to claim all of the elements of a winning set $F \in \mathcal F$, while the goal of the Breaker is to prevent this from happening. One can consider the situations where the Maker or the Breaker has the first move, and moving this advantage from one player to another has the potential to significantly change the outcome of the game. See \cite{HKSS} for more background on Maker-Breaker games, as well as the more general topic of Positional Games.

In this paper, we study the Maker-Breaker Domination game on a finite graph $G=(V,E)$. This game was introduced by Duchêne, Gledel, Parreau and Renault \cite{DGPR}, based closely on earlier work of Bre\v{s}ar, Klav\v{z}ar and Rall \cite{BKR}. Recall that $U \subset V$ is a \textit{dominating set} if
\[
    \forall \, v \in V \setminus U, \, \exists \, u \in U \text{ such that } uv \in E.
\]
The two players in the Maker-Breaker Domination game are called the \textit{Dominator} and the \textit{Staller}, and the players take turns to choose a previously unclaimed vertex from $V$. The goal of the Dominator is to claim all of the elements of a dominating set, and the goal of the Staller is to stop the Dominator from achieving this. One typically thinks of the Dominator as the Maker and the Staller as the Breaker, in which case, the family of winning sets is the family $\mathcal F$ consisting of all dominating sets. It is also possible to view the Staller as the Maker by adjusting the definition of $\mathcal F$ suitably, as we see later in this paper. The game ends when either the Dominator has claimed a dominating set (in which case the Dominator wins), or when all vertices have been claimed and the Dominator has not claimed a dominating set (in which case the Staller wins). 

In particular, each game has exactly one winner. The following table is used to introduce the shorthand notation that we frequently use to describe the outcome on a given graph $G$.
\begin{center}
    \begin{tabular}{p{2cm}p{11cm}}
        \toprule
        Notation & Meaning \\
        \midrule
        $o(G)= \mathcal D$ & The Dominator has a strategy which ensures victory going second. \\
        $o(G)= \mathcal N$ & The Dominator has a strategy which ensures victory going first, and the Staller has a strategy which ensures victory going first. \\
        $o(G)= \mathcal S$ & The Staller has a strategy which ensures victory going second. \\
        \bottomrule
    \end{tabular}
\end{center}

It was shown in \cite{DGPR} that each graph $G$ has exactly one of these three possible outcome statuses. For ease of language and explanation, we shorten the phrasing above. For example, we say that \textit{the dominator wins going second} if $o(G)= \mathcal D$. We use the notation $o(G) \geq \mathcal N$ to say that the Dominator wins going first (while giving no information about who wins when the Dominator has the second move) and the notation $o(G) \leq \mathcal N$ is used to describe the case when the Staller wins going first (while giving no information about who wins when the Staller has the second move).

For example, it is not difficult to show that $o(C_n) \geq \mathcal N$ for all $n \geq 3$. Indeed, one can show this with a \textit{pairing strategy} (see also \cite{DGPR}).
If $n$ is even, then the Dominator can pair off neighbouring vertices of the cycle and proceed to ensure that at least one element from each pair is claimed.
If $n$ is odd, then the Dominator chooses a first vertex arbitrarily and then has a similar pairing strategy for the remaining vertices (see \Cref{fig:C7}).
Moreover, it is not much more difficult to check that the Dominator wins on $C_n$ even going second (see also \cite{DGPR}).

\begin{figure}[ht]
    \centering
    \begin{tikzpicture}
        \foreach \c [count=\i] in {colO,colO,,colR,colR,colB,colB}
        {
            \node[vertex,fill=\c] (v\i) at ({360/7*(\i+4)+90}:2) {};
        }
        \foreach \i [remember=\i as \io (initially 7)] in {1,2,...,7}
        {
            \draw[edge] (v\i)--(v\io);
        }
    \end{tikzpicture}

    \caption{This colouring of a $C_7$ shows a pairing strategy to ensure the Dominator wins on this graph. After choosing the black vertex on the first move, the remaining moves are direct responses to the previous Staller move. If the Staller chooses a red vertex, the Dominator replies by choosing the other red vertex, and the same for the other colours. All vertices are guaranteed to be dominated via this strategy, and so we can write $o(C_7) \geq \mathcal N$.}
    \label{fig:C7}
\end{figure}
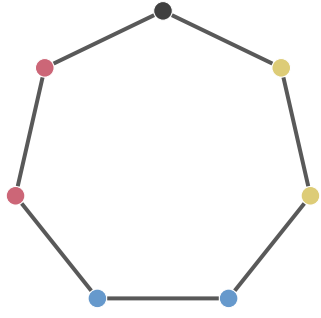

On the other hand, consider the graph $G=P_3 \cup P_3$ formed by two disjoint copies of $P_3$. The Staller has a winning strategy for this graph going second (see \Cref{fig:p3p3}). We therefore write $o(P_3 \cup P_3)= \mathcal S$.

\begin{figure}[ht] 
    \centering
    \begin{tikzpicture}[]
      \node[vertex,label={[labelsty]above:$u_1$}] (a1) at (0,0) {};
      \node[vertex,label={[labelsty]above:$u_2$}] (a2) at (2,0) {};
      \node[vertex,label={[labelsty]above:$u_3$}] (a3) at (4,0) {};
    
      \draw[edge] (a1) -- (a2) (a2) -- (a3);
    
      \node[vertex,label={[labelsty]below:$v_1$}] (b1) at (0,-1.5) {};
      \node[vertex,label={[labelsty]below:$v_2$}] (b2) at (2,-1.5) {};
      \node[vertex,label={[labelsty]below:$v_3$}] (b3) at (4,-1.5) {};
    
      \draw[edge] (b1) -- (b2) (b2) -- (b3);
    \end{tikzpicture}
    
    \caption{The Staller wins on $P_3+P_3$ going second as follows. After the Dominator chooses a first vertex (without loss of generality, we may assume it is a vertex $u_i$), the Staller responds by choosing $v_2$. For the second move, the Staller chooses whichever of $v_1$ and $v_3$ is unclaimed. This traps one of the end vertices of the second path, and ensures that there is a vertex which is not dominated by the Dominator.}
    \label{fig:p3p3}
\end{figure}
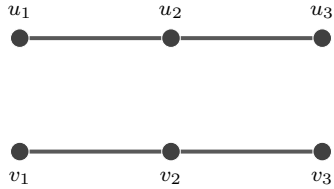

A natural and general problem is to classify the graphs which have each possible outcome. However, obtaining such a classification appears to be a very difficult problem; determining the outcome of the game on a given graph is PSPACE-complete (see \cite{DGPR}). The examples above give some indication that the outcome of the Maker-Breaker Domination game is heavily dependent on the structure (i.e. the edge distribution) of the graph $G$. In particular, it appears that well-connected graphs offer an advantage to the Dominator. Meanwhile, low degree vertices are helpful for the Staller. The main aim of this paper is to study the relationship between the minimum degree of a graph $G$ and the outcome of the game on $G$. Let $\delta(G)$ denote the minimum degree of $G$.

A naive first question to ask is: do there exist graphs $G$ with arbitrarily large minimum degree $d$ such that $o(G)= \mathcal S$? The answer to this question is yes, and a beautiful construction of Gledel (private communication) shows that, for every $d \in \mathbb N$, there exists a graph $G$ with $2^{d+1}+2d$ vertices such that the Staller wins going second. See \Cref{sec:g-construction} for more details concerning this construction. On the other hand, it was observed in \cite{DGPR} that the Erd\H{o}s-Selfridge Theorem (see \Cref{cor:ES}) implies that the Dominator wins going first on any graph with minimum degree $d$ and fewer than $2^{d+1}$ vertices.

The question that we seek to answer in this paper is the following; what is the size (in terms of the number of vertices) of the smallest graph with minimum degree $d$ such that the Staller wins going first. Let $\beta(d)$ denote the answer to this question, that is,
\[
    \beta(d):= \min \{ n \in \mathbb N:  \exists \,  G=(V,E), \,|V|=n, \delta(G)=d,\, o(G)= \mathcal S \}.
\]
The previous paragraph gives
\begin{equation} \label{betarange}
    2^{d+1} \leq \beta(d) \leq 2^{d+1}+2d.  
\end{equation}The aim of this paper is to determine the precise value of $\beta(d)$ for some small values of $d$. Our main result shows that the upper bound in \Cref{betarange} is optimal for $0 \leq d \leq 3$.

\begin{theorem} \label{thm:beta} 
    $\beta(0)=2$, $\beta(1)=6$, $\beta(2)=12$, $\beta(3)=22$.
\end{theorem}

The proofs that $\beta(0)=2$ and $\beta(1)=6$ are straightforward and can be done by hand. Proving that $\beta(2)=12$ could also still be done by hand, however, the computational tools we develop in order to determine $\beta(3)$ deal with $\beta(2)$ quite straightforwardly, and the manual proof requires quite some work, so we give only the computer assisted proof that $\beta(2)=12$.

We develop some theory in this paper which gives some structural information about graphs for which $o(G)= \mathcal S$, namely that such graphs have what we call an \textit{RSI decomposition}, see \Cref{section:RSI_theory}. The proof of $\beta(3)=22$ does then heavily rely on computer-aided computations. Breaking up graphs into its \textit{RSI decomposition} allows us to greatly reduce the required computing resources.

One may notice that all values of $\beta(d)$ that are determined by \Cref{thm:beta} align with the upper bound from \eqref{betarange}, which does not rule out the possibility that Gledel's construction gives the best possible bound for all $d$. On the other hand, we are able to give a very small improvement to the trivial lower bound from \eqref{betarange}. The lower bound $\beta(d) \geq 2^{d+1}+1$ was implicit in the work of \cite{BDGLP}, and this bound can be nudged forward slightly.

\begin{restatable}{theorem}{plusone}\label{thm:plusone} 
    For all $d \in \mathbb N$,
    $\beta(d) \geq 2^{d+1}+2$.
\end{restatable}

We also study the same problem for the case when the Staller has the first move. To this end, define
\[
    \beta'(d):= \min \{ n \in \mathbb N:  \exists \,  G=(V,E), \,|V|=n, \delta(G)=d,\, o(G) \leq  \mathcal N \}.
\]

It turns out $\beta'(d+1)$ can be easily calculated using $\beta(d)$. We prove the following result.

\begin{theorem} \label{thm:beta'+1}
    For all integers $d \geq 0$,
    \[
        \beta'(d+1)= \beta(d)+1.
    \]
\end{theorem}

\section{Notation and Preliminaries}

In this paper all graphs are finite and simple, that is, they are undirected and have no multiple edges or loops. For a graph $G = (V,E)$, we denote by $V(G):=V$ and $E(G) := E$ the vertex set and edge set of $G$, respectively. For two distinct vertices $a, b \in V(G)$, we denote by $ab = ba := \{a,b\}$ the edge between $a$ and $b$, and furthermore, if $A, B \subseteq V(G)$, then $E_G(A,B) = \{ab \in E(G) : a \in A, b\in B\}$ and $e_G(A,B) = |E_G(A,B)|$. In case $A = B$, we use the shorthand notation $E_G(A) = E_G(A,A)$ and $e_G(A) = e_G(A, A)$. In case $G$ is known from context, we drop the subscript $G$. We write $G = (A, B, E)$ to mean that the graph $G = (A \cup B, E)$ is bipartite with bipartition $(A, B)$, that is, $A,B \subseteq V(G)$ partition the vertex set $V(G)$ and $E = E(G) = E(A,B)$.

For a vertex $v \in V(G)$, denote by $\deg_G(v)$, $N_G(v)$, and $N_{G}[v] = N_G(v) \cup \{v\}$ the degree, the neighbourhood, and the closed neighbourhood of $v$, respectively. In case $G$ is clear from context, we drop the corresponding subscript in these notations. If $V(G) = \{v_1,\dots, v_n\}$ with $n = |V(G)|$ such that $\deg(v_1) \leq \dots \leq \deg(v_n)$, then $(\deg(v_1), \dots, \deg(v_n))$ is the \emph{degree sequence} of $G$. If $\deg(v_1) = \deg(v_n) = d$, then the graph $G$ is called $d$-regular. The minimum degree of $G$ is denoted by $\mindeg(G) = \deg(v_1)$. A vertex of degree zero is called \emph{isolated}, and the set of all isolated vertices of a graph $G$ is denoted by $\Iso(G)$. Again, we use the lower case $\iso(G) = |\Iso(G)|$ to denote the number of isolated vertices in $G$. 

A graph $G' = (V',E')$ is a \emph{subgraph} of the graph $G = (V,E)$ if $V' \subseteq V$ and $E' \subseteq E$. In case $V' = V$, we say that $G'$ is a \emph{spanning subgraph} of $G$. A \emph{connected component} of $G$ is a maximal set $C\subseteq V$ where every pair $a,b \in C$ of vertices is connected by a path. For a subset $A \subseteq G$, the \emph{induced subgraph of $G$ on the set $A$} is the subgraph of $G$ with vertex set $A$ and maximal edge set. It is denoted by $G[A]$, and in case $A = V(G) \setminus \{v\}$ for some $v \in V(G)$, we also write $G-v = G[A]$. If $G$ and $H$ are graphs with disjoint vertex sets, then the \emph{disjoint union of $G$ and $H$}, denoted by $G\cup H$, is the graph on $V(G \cup H) = V(G) \cup V(H)$ with $E(G \cup H) = E(G) \cup E(H)$. For positive integers $s,t$, we use the usual notation $K_s$, $K_{s,t}$, and $C_s$ for a complete graph on $s$ vertices, a complete bipartite graph with parts of sizes $s$ and $t$, and a cycle graph on $k$ vertices. The graph $K_{1,t}$ is also called \emph{star}, and a maximum degree vertex of a star is called \emph{centre} of the star.

For a graph $G$ and disjoint subsets $D, S \subseteq V(G)$, we call $(G,D,S)$ a \emph{game state on $G$} and think of the vertices in $D$ and $S$ to be the vertices claimed by the Dominator and the Staller, respectively. In particular, it makes sense to talk about the MBD-game on a game state. We already saw an informal definition of the MBD-game in \Cref{sec:introduction}, but especially for the later \Cref{algorithm:outcome_game_state}, it is insightful to have a more formal recursive definition. Let us use $\mathcal D$ and $\mathcal S$ as synonyms for the players Dominator and Staller and define the order $\leq$ on $\{\mathcal D, \mathcal S\}$ by $\mathcal S \leq \mathcal D$.

\begin{definition}\label{definition:recursive_outcome}
    Let $(G,D,S)$ be a game state, and let $p \in \{\mathcal D, \mathcal S\}$. We define the \emph{outcome $o_p(G,D,S) \in \{\mathcal D, \mathcal S\}$ of the game state $(G,D,S)$ and the first player $p$} as follows. If $D \cup S = V(G)$, then $o_p(G,D,S) = \mathcal S$ if and only if there is a vertex $v \in V(G)$ with $N_{G}[v] \subseteq S$. On the other hand, if $D \cup S \neq V(G)$, then 
    \begin{align*}
        &o_{\mathcal D}(G,D,S) = \max\{o_{\mathcal S}(G,D\cup \{v\},S) : v \in V(G) \setminus (D \cup S)\} \text{ and}\\
        &o_{\mathcal S}(G,D,S) = \min\{o_{\mathcal D}(G,D,S\cup \{v\}) : v \in V(G) \setminus (D \cup S)\}.
    \end{align*}
    Here $\min$ and $\max$ are with respect to the order $\leq$ on $\{\mathcal D, \mathcal S\}$. 
    In case $D = S = \emptyset$, we drop them in the notation and write $o_p(G) = o_p(G,D,S)$.
\end{definition}

We shall also use the more common notion of outcome $o(G) \in \{\mathcal D, \mathcal N, \mathcal S\}$ defined for a graph $G$ as follows. 
\begin{align*}
    o(G) = \mathcal D \ &\Leftrightarrow \ o_{\mathcal S}(G) = \mathcal D\\
    o(G) = \mathcal N \ &\Leftrightarrow \ o_{\mathcal D}(G) = \mathcal D \text{ and } o_{\mathcal S}(G) = \mathcal S\\
    o(G) = \mathcal S \ &\Leftrightarrow \ o_{\mathcal D}(G) = \mathcal S
\end{align*}
If one of the players wins on a graph going second, they also win going first (cf. \cite[Proposition 2.1.6]{HKSS}). This means that this notion of outcome of a graph is well-defined. As before, it is convenient to introduce the linear ordering $\mathcal S \leq \mathcal N \leq \mathcal D$ on $\{\mathcal S,\mathcal N, \mathcal D\}$ so that we can, for example, write $o(G) \geq \mathcal N$ to mean that the Dominator has a winning strategy on $G$ as the first player.

\subsection*{Computational Methods}
Part of this work relies on computer-aided computations with graphs \cite{repo}. 
The graphs of interest are generated using \nauty\ \cite{McKayPiperno2014}, and the code used to process these graphs is contained mainly in the files \texttt{graph.cpp}, \texttt{game\_state.cpp}, and \texttt{pccgs.cpp} in \cite{repo}. All computations that are needed in this work are executed via the bash script \texttt{generate\_and\_check.sh}, which is referenced in the proof of any theorem or lemma established through brute-force computation. After running \texttt{generate\_and\_check.sh}, details on whether the computations finished successfully and on how long they took can be found in log files stored in the \texttt{logs} directory. Running the whole bash script takes roughly 147 CPU-days.
From now on, file names in teletype font are  referring to \cite{repo}. For detailed documentation and instructions on running the code, the reader is referred to \texttt{README.md} and the inline comments in the code. 

\section{An upper bound construction}\label{sec:g-construction}

In this section, we derive the upper bound for $\beta(d)$. This is given by a general graph construction introduced to us by Valentin Gledel.

\begin{theorem}\label{proposition:g-construction}
    For every $d \in \mathbb N$, there exists a graph $G$ with $2^{d+1}+2d$ vertices such that $o(G)= \mathcal S$.
\end{theorem}

\begin{proof}
    We construct a bipartite graph $G=(X,Y,E)$ with
    \[
        |X|=2d, \,\ |Y|=2^{d+1}
    \]
    such that the Staller wins the game on $G$ going second. The elements of $X$ and $Y$ each come in pairs; let $a$, $b$, $A$ and $B$ be symbols for denoting vertices and let
    \[
        X=\{a_1,b_1,\dots , a_d,b_d \},
    \]
    and let
    \[
        Y=\bigcup_{s \in \{a,b\}^d} \{A_s,B_s \}.
    \]
    Vertices in $Y$ are indexed by binary sequence $s \in \{a,b\}^d$, so we can define the edge set $E$ as follows: for each sequence $s$, we draw an edge between the corresponding elements of $X$ and both vertices $A_s$ and $B_s$. For example, the sequence $s=(a,a,\dots,a)$ gives rise to a complete bipartite subgraph $K_{d,2}$ between the vertex sets $\{ a_1, \dots, a_d \}$ and $\{A_s,B_s\}$.

    A pairing strategy is used to show that the Staller wins. The Staller decomposes the vertex set into disjoint pairs 
    \[
        (a_1,b_1), \dots , (a_d,b_d), (A_{s_1}, B_{s_1}), \dots, (A_{s_{2^d}}, B_{s_{2^d}}).
    \]
    After each move of the Dominator, the Staller responds by choosing the other vertex in the pairing. At the end of the game, there is some sequence $s \in \{a,b\}^d$ such that the Staller has chosen all from $X$ corresponding to this sequence, and the Staller has also claimed a vertex $x \in \{A_s, B_s \}$. It follows that the Staller has claimed $x$ and all of its neighbours, and thus this pairing strategy guarantees that the Staller wins.
    See \Cref{fig:g-construction} for an example.
\end{proof}

\begin{figure}[ht]
    \centering
    \begin{tikzpicture}
        \foreach \s [count=\i] in {aa,ab,ba,bb}
        {
            \node[vertex,label={[labelsty]left:$A_{\s}$}] (A\s) at (-4,(\i*0.8-2) {};
            \node[vertex,label={[labelsty]right:$B_{\s}$}] (B\s) at (4,\i*0.8-2) {};
        }
        \node[vertex,label={[labelsty]180-40:$a_1$}] (a1) at (0,0.5) {};
        \node[vertex,label={[labelsty]180-20:$a_2$}] (a2) at (0,1.5) {};
        \node[vertex,label={[labelsty]180+40:$b_1$}] (b1) at (0,-0.5) {};
        \node[vertex,label={[labelsty]180+20:$b_2$}] (b2) at (0,-1.5) {};
        \foreach \s/\t/\col [count=\k] in {{a,a}/aa/colR,{a,b}/ab/colG,{b,a}/ba/colO,{b,b}/bb/colB}
        {
            \foreach \v [count=\i] in \s
            {
                \ifnum\k<2
                    \pgfmathparse{1}\let\op=\pgfmathresult
                \else
                    \pgfmathparse{0.25}\let\op=\pgfmathresult
                \fi
                \begin{scope}[opacity=\op]
                    \draw[edge,\col] (\v\i)--(A\t);
                    \draw[edge,\col] (\v\i)--(B\t);
                \end{scope}
            }
        }

        \foreach \x in {a1,a2,Baa,Bab,Aba,Bbb}
        {
            \node[crossvertex] at (\x) {};
        }
    \end{tikzpicture}

    \caption{A partial illustration of the construction for \Cref{proposition:g-construction} in the case when $d=2$. We draw the edges corresponding to two of the four binary sequences of length $2$ (since the picture becomes too crowded to be helpful when all of the edges are drawn). The crosses indicate vertices that the Staller chooses during the game. In this case, the vertex $B_{aa}$ is trapped by the Staller; the Staller chooses this vertex and all of its neighbours.}
    \label{fig:g-construction}
\end{figure}
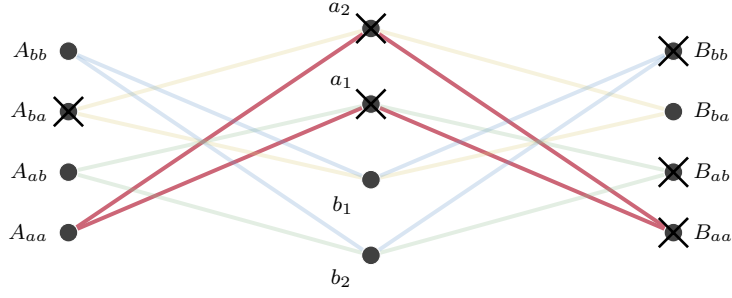
This immediately implies the following upper bound on $\beta(d)$.
\begin{corollary}\label{cor:g-bound}
    $\beta(d) \leq 2^{d+1}+2d$.
\end{corollary}

\section{The Erd\H{o}s-Selfridge Theorem}\label{section:ES}

This paper makes heavy use of a theorem of Erd\H{o}s-Selfridge on winning strategies of Maker-Breaker games.

\begin{theorem}[{Erd\H{o}s-Selfridge Theorem \cite{ErdösSelfridge}}] \label{thm:ES}
    Consider a Maker-Breaker game with a finite base set $X$ such that $\mathcal F \subset \mathcal P(X)$ is the set of all winning sets for the Maker. If the Breaker goes first and 
    \[
        \sum_{F \in \mathcal F} 2^{-|F|} < 1,
    \]
    then the Breaker has a winning strategy. If the Breaker goes second and 
    \[
        \sum_{F \in \mathcal F} 2^{-|F|} < 1/2,
    \]
    then the Breaker has a winning strategy.
\end{theorem}

The following corollary rephrases \Cref{thm:ES} in the context of the Maker-Breaker Domination game.

\begin{corollary} \label{cor:ES}
    Let $G=(V,E)$ be a graph such that
    \[
        \sum_{v \in V} 2^{-(\deg(v)+1)} < 1.
    \]
    Then $o(G) \geq \mathcal N$.
\end{corollary}

Given a graph $G$, we define the \textit{potential of $G$} to be the quantity
\[
    pot(G):= \sum_{v \in V} 2^{-(\deg(v)+1)}.
\]
With this notation, \Cref{cor:ES} states that
\[
    pot(G) < 1 \implies o(G)\geq \mathcal N.
\]

One may think of the potential $pot(G)$ as “the initial potential for the Staller to win”. However, we can also regard the potential as a dynamic quantity that changes with each move. The family of winning sets is changing with each move. The idea here is that, after each move in a Maker Breaker game, we may reconsider the game as if it were to start from the current situation and then apply \Cref{thm:ES} if possible.

For a game state $(G,D,S)$ and a vertex $v \in V(G)$, we define the \emph{game state degree of $v$} to be
\begin{align*}
    \gsdeg_{(G,D,S)}(v) = \begin{cases}
        \infty, &\text{if } N[v] \cap D \neq \emptyset,\\
        \left|N[v] \setminus S\right| &\text{otherwise}.
    \end{cases}
\end{align*}
If the game state is clear from the context, we may drop the subscript. The \emph{Erd\H{o}s-Selfridge potential}, or just \emph{potential} of the game state $(G,D,S)$, is given by
\begin{align*}
    \pot(G,D,S) = \sum_{v \in V(G)} 2^{-\gsdeg(v)}.
\end{align*}
Here we use the convention that $2^{-\infty} = 0$. We often make use of the following version of the Erd\H{o}s-Selfridge Theorem, specified for game states.
\begin{theorem}[Erd\H{o}s-Selfridge Theorem for game states]\label{theorem:ES_for_gs}
    Let $(G,D,S)$ be a game state. If $\pot(G,D,S) < 1$, then $o_{\mathcal D}(G,D,S) = \mathcal D$. If $\pot(G,D,S) < \frac 12$, then $o_{\mathcal S}(G,D,S) = \mathcal D$. 
\end{theorem}
\begin{proof}
    For a game state $(G,D,S)$ and an unclaimed vertex $v \in V(G) \setminus (D \cup S)$, define the potential of $v$ to be 
    \[\pot_{(G,D,S)}(v) := \sum_{u\in N[v]} 2^{-\gsdeg(v)}.\]
    We prove that the Dominator can guarantee to win by always picking the vertex with the highest potential.
    
    Note that the Dominator (respectively Staller) picking the vertex $v$ in the game state $(G,D,S)$ decreases (respectively increases) the potential of the game state by exactly $\pot_{(G,D,S)}(v)$. In other words
    \begin{align}\label{align:pot_change}
        \pot_{(G,D,S)}(v) = \pot(G,D,S) - \pot(G,D \cup \{v\},S) = \pot(G,D,S \cup \{v\}) - \pot(G,D,S). 
    \end{align}
    Also note that if $(G,D_1,S)$ and $(G,D_2,S)$ are two game states with $D_1 \subseteq D_2$ and $v \in V(G) \setminus (D_2 \cup S)$, then 
    \begin{align}\label{align:pot_monotone}
        \pot_{(G,D_1,S)}(v) \geq \pot_{(G,D_2,S)}(v).
    \end{align}
    If $d \in V(G) \setminus (D \cup S)$ is a vertex with maximal potential in the game state $(G,D,S)$ and $s \in V(G) \setminus (D \cup \{d\}\cup S)$ is any vertex, then it follows from \Cref{align:pot_change} and \Cref{align:pot_monotone} that
    \begin{align*}
        \pot(G,D\cup \{d\}, S \cup \{s\}) &= \pot(G,D,S) - \pot_{(G,D,S)}(d) + \pot_{(G,D\cup \{d\},S)}(s)\\ 
        &\leq \pot(G,D,S).
    \end{align*}
    Putting things together, we see that if $\pot(G,D,S) < 1$ and the Dominator sticks to his strategy as first player, he can guarantee that the potential is always less than $1$. At the end of the game, all the vertices are claimed, and the potential has to be a non-negative integer. The potential at the end of the game is thus $0$, which means that the Dominator won. If $\pot(G,D,S) < \frac12$ and the Staller is the first player, she can at most double the potential in her first move. After that we are again in the case of $\pot(G,D,S) < 1$ and the Dominator being the first player. 
\end{proof}
A more detailed proof of a version of the result above for general Maker-Breaker games on hypergraphs can be found in \cite[Theorem 2.3.3]{HKSS}. For us, it is important that the Dominator usually wants to decrease the potential, while the Staller usually wants to increase it.

As a first application of \Cref{theorem:ES_for_gs}, we show the following.

\begin{proposition}\label{proposition:beta'(d+1)=beta(d)+1_large_d}
    For all integers $d\geq 3$, we have $\beta'(d+1)=\beta(d)+1$.
\end{proposition}

It is convenient to split off two lemmas. 
\begin{lemma}\label{lemma:delete_huge_deg_vertex_if_Staller_goes_first}
Let $G$ be a graph on $n$ vertices, and let $v \in V(G)$ with $\deg(v) \geq \frac n2$. Then 
    \[o_{\mathcal D}(G - v) = o_{\mathcal D} (G, \emptyset, \{v\}).\]
\end{lemma}
\begin{proof}
    As the first player, the Dominator occupies exactly $\lfloor \frac n2 \rfloor$ vertices at the end of the game, both on $G-v$ and on $(G, \emptyset,\{v\})$. But a set $D \subseteq V(G) \setminus \{v\}$ of size $|D| = \lfloor \frac n2 \rfloor$ is a dominating set of $G$ if and only if it is a dominating set of $G-v$, since $\deg(v) \geq \frac n2$. This shows that a strategy for the Dominator as the first player on $G-v$ is a winning strategy if and only if it is a winning strategy on $(G,\emptyset, \{v\})$. 
\end{proof}

\begin{lemma}\label{lemma:beta'(d+1)_leq_beta(d)+1}
    Let $d$ be a non-negative integer. Then $\beta'(d+1)\leq \beta(d)+1$.
\end{lemma}
\begin{proof}
    Given a graph $G$ on $\beta(d)$ vertices with $\mindeg(G) = d$ and $o(G) = \mathcal S$, we construct a graph $G'$ on $\beta(d)+1$ vertices with $\mindeg(G') = d+1$ and $o(G') \leq \mathcal N$ as follows.
    Let $G'$ have vertex set $V(G') = V(G) \cup \{v\}$ for a new vertex $v \notin V(G)$ and edge set $E(G') = E(G) \cup \{vu : u \in V(G) \}$.
    Certainly, $|V(G')|=\beta(d)+1$ and $\mindeg(G')=d+1$ as $\deg_{G'}(v) = \beta(d)\geq 2^{d+1} \geq d+1$, by \Cref{betarange}. Furthermore, by \Cref{lemma:delete_huge_deg_vertex_if_Staller_goes_first} and because of the assumption $o(G) = \mathcal S$, we get
    \[
        o_{\mathcal D}(G', \emptyset, \{v\}) = o_{\mathcal D}(G) = \mathcal S.
    \]
    This shows $o(G') \leq \mathcal N$.
\end{proof}

\begin{proof}[Proof of \Cref{proposition:beta'(d+1)=beta(d)+1_large_d}]
    By \Cref{lemma:beta'(d+1)_leq_beta(d)+1}, it remains to show that $\beta'(d+1) > \beta(d)$. Let $G'$ be a graph on $n\leq \beta(d)$ vertices with $\mindeg(G')=d+1$.
    We have to prove that $o(G') = \mathcal D$. 
    Let $v \in V(G')$ be the vertex which the Staller chooses first, and let $d_1=\deg(v)$.
    We distinguish two cases depending on whether $d_1 \geq \frac n2$ or $d_1 < \frac n2$.
    
    First, if $d_1 \geq \frac n2$, note that $G = G'-v$ satisfies $\mindeg(G)\geq d$ and $|V(G)| = n-1 < \beta(d)$, so that $o_{\mathcal D}(G) = \mathcal D$.
    By \Cref{lemma:delete_huge_deg_vertex_if_Staller_goes_first}, we also get $o_{\mathcal D}(G', \emptyset, \{v\}) = \mathcal D$.
    
    Second, if $d_1 \leq \lceil \frac n2 \rceil -1$, we simply use the Erd\H{o}s-Selfridge Theorem (\Cref{theorem:ES_for_gs}) for the game state $(G', \emptyset, \{v\})$. By \Cref{cor:g-bound}, we have $n\leq \beta(d) \leq 2^{d+1}+2d$ so that
    \begin{align*}
        \pot(G',\emptyset,\{v\})&=\sum_{u\in V(G')\setminus N_{G'}[v]} 2^{-\deg_{G'}(u)-1}+\sum_{u\in N_{G'}[v]} 2^{-\deg_{G'}(u)}\\
        &\leq \frac{n-d_1-1}{2^{d+2}}+ \frac{d_1}{2^{d+1}} + \frac1{2^{d_1}} \leq \frac12 + \frac{d_1+2d-1}{2^{d+2}} + \frac1{2^{d_1}}.
    \end{align*}
    But for $d \geq 3$ and $d+1 \leq d_1 \leq \lceil \frac n2 \rceil -1 \leq 2^d+d-1$, this last expression is less than $1$. For $d+3 \leq d_1 \leq 2^d+d-1$, this follows from 
    \[
        \frac12 + \frac{d_1+2d-1}{2^{d+2}} + \frac1{2^{d_1}} \leq \frac12 + \frac{2^d+3d-2}{2^{d+2}} + \frac1{2^{d+3}}\leq  \frac34 + \frac{6d-3}{2^{d+3}} <1
    \]
    and for $d+ 1 \leq d_1 \leq d+2$ from 
    \[
        \frac12 + \frac{d_1+2d-1}{2^{d+2}} + \frac1{2^{d_1}} \leq \frac12 + \frac{3d+1}{2^{d+2}} + \frac1{2^{d+1}}<1.
    \]
    The Dominator now wins on $(G,\emptyset, \{v\})$ going first because of \Cref{theorem:ES_for_gs}. Since $v \in V(G')$ was arbitrary, we get $o(G') = \mathcal D$.
\end{proof}

In \Cref{section:low_surplus} and for computing $\beta(d)$ and $\beta'(d+1)$ for $d \in \{1,2\}$ (see the forthcoming \Cref{proposition:tiny_betas}), we use computers to check all relevant graphs using the following \Cref{algorithm:outcome_game_state}. At each game state, it first checks two winning conditions: It checks whether the Staller has already claimed a winning set, and it checks whether the potential is small enough, such that it is guaranteed the Dominator has a winning strategy. In all other cases, it recursively goes over all continuations of the game and decides if there is a winning move for the player who moves next.

Generally, deciding the outcome of a graph is PSPACE-complete (cf. \cite{DGPR}), and in fact, for our purposes, this approach only works for very small graphs. The implementation of \Cref{algorithm:outcome_game_state} we use can be found in \texttt{game\_state.cpp}. There we additionally use the heuristic that claiming vertices $v$ with high potential $\pot_{(G,D,S)}(v)$, as defined in the proof of \Cref{theorem:ES_for_gs}, is typically better than claiming vertices with low potential. Accordingly, the loops in \lineref{algorithm_line:gs_outcome_Dominator_forloop} and \lineref{algorithm_line:gs_outcome_Staller_forloop} of \Cref{algorithm:outcome_game_state} iterate over the vertices in decreasing order of potential. Note that \Cref{algorithm:outcome_game_state} closely resembles the recursive \Cref{definition:recursive_outcome}.

\begin{algorithm}[ht]
    \caption{\textsc{Outcome}$((G,D,S),p)$}
    \label{algorithm:outcome_game_state}
    \begin{algorithmic}[1]
        \REQUIRE A game state $(G,D,S)$ and a first player $p \in \{\mathcal D, \mathcal S\}$
        \ENSURE Returns $o_p(G,D,S)$
    
        \FORALL{$v \in V(G)$}
            \STATE \textbf{if} $\gsdeg_{(G,D,S)}(v) = 0$ \textbf{then} \textbf{return} $\mathcal S$\label{algorithm_line:gs_outcome_Staller_already_won}
        \ENDFOR
        \IF{$p = \mathcal D$}
            \STATE \textbf{if} $\pot(G,D,S) < 1$ \textbf{then} \textbf{return} $\mathcal D$\label{algorithm_line:gs_outcome_Dominator_potcheck}
            \FORALL{$v \in V(G) \setminus (D \cup S)$} \label{algorithm_line:gs_outcome_Dominator_forloop}
                \STATE \textbf{if} $\textsc{Outcome}((G,D \cup \{v\}, S), \mathcal S) = \mathcal D$ \textbf{then} \textbf{return} $\mathcal D$\label{algorithm_line:gs_outcome_Dominator_inside_forloop}
            \ENDFOR
            \STATE \textbf{return} $\mathcal S$\label{algorithm_line:gs_outcome_Dominator_no_winning_vertex}
        \ELSIF{$p = \mathcal S$}
            \STATE \textbf{if} $\pot(G,D,S) < \tfrac{1}{2}$ \textbf{then} \textbf{return} $\mathcal D$\label{algorithm_line:gs_outcome_Staller_potcheck}
            \FORALL{$v \in V(G) \setminus (D \cup S)$} \label{algorithm_line:gs_outcome_Staller_forloop}
                \STATE \textbf{if} $\textsc{Outcome}((G,D,S \cup \{v\}), \mathcal D) = \mathcal S$ \textbf{then} \textbf{return} $\mathcal S$\label{algorithm_line:gs_outcome_Staller_inside_forloop}
            \ENDFOR
            \STATE \textbf{return} $\mathcal D$
        \ENDIF
    \end{algorithmic}
\end{algorithm}

\begin{proof}[Proof of correctness of \Cref{algorithm:outcome_game_state}]
    Let $(G,D,S)$ be a game state, and let $p\in \{\mathcal D,\mathcal S\}$ be a first player. If there is $v \in V(G)$ with $\gsdeg(v) = 0$, then all the vertices in $N_{G}[v]$ are claimed by the Staller, and the Staller wins on $(G,D,S)$ regardless of the first player. In this case, we indeed return $\mathcal S$ in \lineref{algorithm_line:gs_outcome_Staller_already_won}. From now on, we may assume that $\gsdeg(v) > 0$ for all $v \in V(G)$.
    
    We proceed by induction on $n := |V(G) \setminus (D \cup S)|$. For the induction base $n = 0$, note that $\pot(G,D,S) \in \Z_{\geq 0}$ and thus $\pot(G,D,S) = 0$ since $\gsdeg(v) > 0$ for all $v \in V(G)$. This means that the Dominator wins on $(G,D,S)$ regardless of the first player, and indeed the algorithm returns $\mathcal D$ in \lineref{algorithm_line:gs_outcome_Dominator_potcheck} or \lineref{algorithm_line:gs_outcome_Staller_potcheck}. Now assume that $n > 0$ and that \Cref{algorithm:outcome_game_state} matches its specification for game states $(G',D',S')$ with $|V(G') \setminus (D' \cup S')| < n$.
        
    Assume that $p = \mathcal D$. If $\pot(G,D,S) < 1$, then $o_{\mathcal D}(G,D,S) = \mathcal D$ by the Erd\H os-Selfridge Theorem, and indeed we return $\mathcal D$ in \lineref{algorithm_line:gs_outcome_Dominator_potcheck}. If there is a vertex $v \in V(G) \setminus (D \cup S)$ with $o_{\mathcal S}(G, D \cup \{v\}, S) = \mathcal D$, then we should return $\mathcal D$, and indeed, since $|V(G) \setminus (D \cup \{v\} \cup S)| = n-1$, the algorithm in fact returns $\mathcal D$ in \lineref{algorithm_line:gs_outcome_Dominator_inside_forloop} by the induction hypothesis. If there is no such $v$, we return $\mathcal S$ in \lineref{algorithm_line:gs_outcome_Dominator_no_winning_vertex}, and indeed $o_{\mathcal D}(G, D, S) = \mathcal S$ in that case. The case $p = \mathcal S$ is analogous to the case $p = \mathcal D$. 
\end{proof}

For $d \leq 2$, computing $\beta(d)$ and $\beta'(d+1)$ using \nauty\ and our implementation of \Cref{algorithm:outcome_game_state} brute-force is enough.

\begin{proposition}\label{proposition:tiny_betas}
    For $d \in \{0, 1, 2\}$ we have 
    \begin{align*}
        \beta(d) = 2^{d+1} + 2d \text{ and } \beta'(d+1) = 2^{d+1} + 2d+1.
    \end{align*}
\end{proposition}
\begin{proof}
    For $d = 0$, this is immediate, so let us assume $d \in \{1,2\}$. 
    The upper bounds come from \Cref{cor:g-bound} and \Cref{lemma:beta'(d+1)_leq_beta(d)+1}. For the lower bounds for $\beta(d)$ with $d \in \{1,2\}$, we just check whether the Dominator wins on all graphs of minimum degree $d$ and $2^{d+1} \leq n \leq 2^{d+1}+2d-1$ vertices going second using \Cref{algorithm:outcome_game_state}. We do not have to check graphs with fewer vertices because of \Cref{theorem:ES_for_gs}. 
    For $\beta(1)$ we need to check 16 graphs and for $\beta(2)$ there are \pgfmathprintnumber{330541583} graphs to be checked, which takes around three hours.
    
    For bounding $\beta'(2)$ and $\beta'(3)$ from below, we proceed analogously, but for the case of minimum degree $3$ and $12$ vertices, we additionally assume the maximum degree to be at most $7$. This assumption is possible because of the following. Assume $G$ is a graph on $12$ vertices of minimum degree at least $3$ and with a vertex $v$ of degree at least $8$. If the Staller chooses $v$ in her first move, the Dominator wins because of \Cref{lemma:delete_huge_deg_vertex_if_Staller_goes_first} and since $\beta(2) = 12$. If the Staller chooses $w \neq v$ in her first move, the Dominator can claim $v$, reducing the potential to at most
    \[
        \pot(G,\{v\},\{w\}) \leq |N[w] \setminus N[v]| \cdot 2^{-3} < \frac 12,
    \]
    so that indeed $o(G) = \mathcal D$.
    
    For details on the computations carried out, we refer the reader to the small graphs section in \texttt{generate\_and\_check.sh}. 
    For $\beta'(2)$ we need to check 53 graphs and for $\beta(2)$ there are \pgfmathprintnumber{19407969744} graphs to be checked which takes around 205 hours of computation time.
\end{proof}

\section{Partition Strategies and RSI-decompositions}\label{section:RSI_theory}

In \Cref{section:low_surplus} and for \Cref{thm:plusone}, we need a slight modification of a partition strategy from \cite{BDGLP}. Let us first briefly summarize the relevant results from \cite{BDGLP}.

Let $G$ be a graph, and let $F$ be a spanning subgraph of $G$. Then $F$ is called \emph{perfect $[1,2]$-factor of $G$}, if every connected component of $F$ is either isomorphic to a single edge $K_2$ or a cycle graph $C_n$ with $n \geq 3$. The Dominator wins on $K_2$ and on cycle graphs $C_n$ going second. This means the Dominator wins on the perfect $[1,2]$-factor $F$ of $G$ going second, since the outcome of a graph can be computed from the outcome of its components, as shown in \Cref{table:outcome_of_disjoint_union}. 
\begin{table}[!ht]
    \centering
    \begin{tabular}{c  c  c  c }
        \toprule
        $o(G \cup H)$ & $o(G) = \mathcal D$ & $o(G) = \mathcal N$ & $o(G) = \mathcal S$ \\
        \midrule
        $o(H) = \mathcal D$ & $\mathcal D$ & $\mathcal N$ & $\mathcal S$ \\
        $o(H) = \mathcal N$ & $\mathcal N$ & $\mathcal S$ & $\mathcal S$ \\
        $o(H) = \mathcal S$ & $\mathcal S$ & $\mathcal S$ & $\mathcal S$ \\\bottomrule
    \end{tabular}
    \caption{Outcome of disjoint union of two graphs}
    \label{table:outcome_of_disjoint_union}
\end{table}

Finally, since $F$ is a spanning subgraph of $G$, we clearly have $o(G) \geq o(F) = \mathcal D$, so that we obtain the following.

\begin{proposition}[{\cite[Lemma 3.1]{BDGLP}}]
    \label{proposition:perfect_1_2-factor_implies_Dominator_wins}
    Let $G$ be a graph. If $G$ admits a perfect $[1,2]$-factor, then $o(G) = \mathcal D$.
\end{proposition}

Much like Tutte's condition for a graph to admit a perfect matching \cite{Tutte1947}, there is a necessary and sufficient condition for a graph to have a perfect $[1,2]$-factor. For a proof of the next theorem, we refer the reader to \cite[Theorem 7.2]{Akiyama_Kano_book}.

\begin{theorem}[Tutte's condition for perfect ${[1,2]}$-factors, \cite{Tutte_1953}]
    \label{theorem:Tutte_condition_for_perfect_1_2-factor}
    A graph $G$ admits a perfect $[1,2]$-factor if and only if $\iso(G-S) \leq |S|$ for all $S \subseteq V(G)$.
\end{theorem}

While the Erd\H{o}s-Selfridge Theorem applies to graphs on few vertices with high degrees, the winning condition for the Dominator coming from having a perfect $[1,2]$-factor is particularly useful for almost regular graphs. The Dominator, for example, wins on regular graphs.

\begin{corollary}[{\cite[Corollary 3.5]{BDGLP}}]\label{corollary:Dominator_wins_on_regular_graphs} 
Regular graphs have perfect $[1,2]$-factors. In particular, the Dominator wins on regular graphs going second. 
\end{corollary}
\begin{proof}
Let $G$ be a regular graph of degree $d$, and let $S \subseteq V(G)$. Then 
\begin{align*}
    \iso(G-S) d = e(S, \Iso(G-S)) \leq d|S|, 
\end{align*}
so that $G$ has a perfect $[1,2]$-factor by \Cref{theorem:Tutte_condition_for_perfect_1_2-factor}.
\end{proof}
The modification of perfect $[1,2]$-factors we need takes into account the extra first move the Dominator has as first player. Instead of requiring every component of the spanning subgraph $F$ to be an edge or a cycle, we may also allow a single star to be among the components of $F$.

\begin{definition}[$v$-factor]\label{definition:v-factors}
    Let $G$ be a graph, and let $F$ be a spanning subgraph of $G$. Assume there is a connected component $F_1$ of $F$ that is isomorphic to a star $K_{1,k}$ with $k \geq 0$, and let $v$ be the centre of the star $F_1$. If $F - V(F_1)$ is a perfect $[1,2]$-factor of $G - V(F_1)$, then $F$ is called \emph{$v$-factor of $G$}.
\end{definition}

It follows from \Cref{proposition:perfect_1_2-factor_implies_Dominator_wins} that if $G$ admits a $v$-factor for a vertex $v \in V(G)$, then the Dominator wins on $(G,\{v\},\emptyset)$ going second and thus on $G$ going first.

The next proposition, \Cref{proposition:v-factor_RSI-decomp}, should be seen as an analogue to \Cref{theorem:Tutte_condition_for_perfect_1_2-factor} for $v$-factors. For a graph $G$ to have a perfect $[1,2]$-factor, we have to forbid subsets $S\subseteq V(G)$ with $\iso(G-S) > |S|$. It turns out that for $G$ to have a $v$-factor, we have to forbid the following structure.

\begin{definition}[RSI-decomposition]\label{definition:RSI_decomp}
    Let $G$ be a graph, let $v$ be a vertex of $G$, and let $R,S,I$ be a partition of $V(G)$. If $v\in R$, $I = \iso(G-S) \setminus \{v\}$, and $|I|>|S|$, then the triple $(R,S,I)$ is called an \emph{RSI-decomposition of $G$ separating $v$}. Furthermore, an RSI-decomposition $(R,S,I)$ is called \emph{proper} if every $s\in S$ has at least two distinct neighbours in $I$. 
\end{definition}

The vertex $v$ is usually a fixed vertex of maximum degree, so we often omit explicitly mentioning $v$ and just speak of an \emph{RSI-decomposition of $G$}. 

\begin{proposition}\label{proposition:v-factor_RSI-decomp}
    Let $G$ be a graph and let $v$ be a vertex of $G$. Then $G$ has a $v$-factor if and only if there is no RSI-decomposition of $G$ separating $v$. 
\end{proposition}
\begin{proof}
    Let $H$ be the unique graph with one vertex of degree $2$ and $4$ vertices of degree $3$. We construct a new graph $G'$ from $G$ as follows. For each $u\in N_G(v)$, join a copy $H_u$ of $H$ to $G-v$ and add the edge $uw_u$, where $w_u$ is the degree $2$ vertex of this copy $H_u$ of $H$. We show that $G$ has a $v$-factor if and only if $G'$ has a perfect $[1,2]$-factor.

    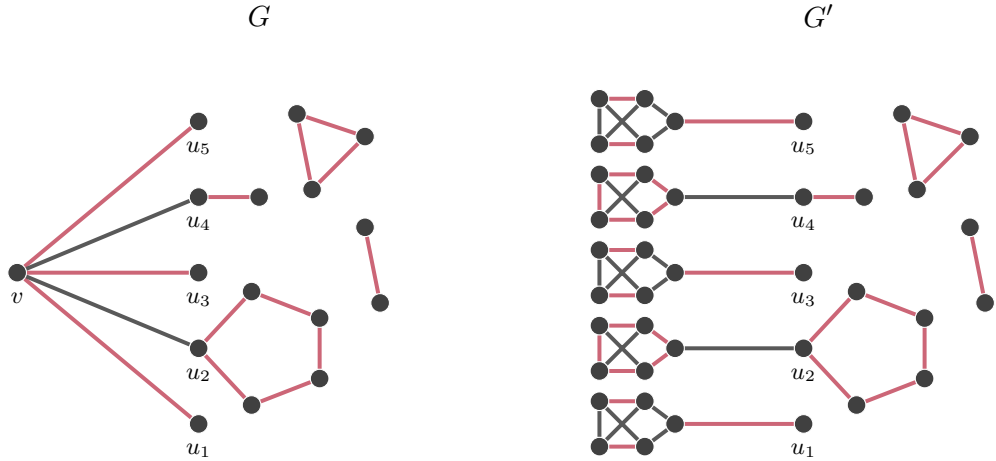
\begin{figure}[ht]
        \centering
        \begin{tikzpicture}
            \begin{scope}
                \node[] at(-0.4,3.4) {$G$};
                \node[vertex, label=below:\footnotesize $v$] (v) at (-3.6,0) {};
                \foreach \x [count=\i] in {-2,-1,...,2} {
                    \node[vertex, label=below:\footnotesize $u_{\i}$] (u\i) at (-1.2,\x) {};
                }    
                \draw[rededge] (v)--(u1);
                \draw[rededge] (v)--(u3);
                \draw[rededge] (v)--(u5);
                \draw[edge] (v)--(u2);
                \draw[edge] (v)--(u4);
        
                \node[vertex] (p1) at (-0.5,-0.25) {};
                \node[vertex] (p2) at (-0.5,-1.75) {};
                \node[vertex] (p3) at (0.4,-0.6) {};
                \node[vertex] (p4) at (0.4,-1.4) {};
                \draw[rededge] (u2)--(p1);
                \draw[rededge] (p1)--(p3);
                \draw[rededge] (p3)--(p4);
                \draw[rededge] (p4)--(p2);
                \draw[rededge] (p2)--(u2);
        
                \node[vertex] (m) at (-0.4,1) {};
                \draw[rededge] (u4)--(m);
        
                \node[vertex] (t1) at (0.3,1.1) {};
                \node[vertex] (t2) at (0.1, 2.1) {};
                \node[vertex] (t3) at (1, 1.8) {};
                \draw[rededge] (t1)--(t2);
                \draw[rededge] (t2)--(t3);
                \draw[rededge] (t3)--(t1);
        
                \node[vertex] (b1) at (1,0.6) {};
                \node[vertex] (b2) at (1.2,-0.4) {};
                \draw[rededge] (b1)--(b2);
            \end{scope}
        
            \begin{scope}[xshift=8cm]
                \node[] at(-1,3.4) {$G'$};
        
                \foreach \x [count=\i] in {-2,-1,...,2}
                {
                    \node[vertex, label=below:\footnotesize $u_{\i}$] (u\i) at (-1.2,\x) {};
                    \node[vertex] (k1\i) at (-2.9,\x) {};
                    \node[vertex] (k2\i) at ($(k1\i)+(-0.4,0.3)$) {};
                    \node[vertex] (k3\i) at ($(k1\i)+(-0.4,-0.3)$) {};
                    \node[vertex] (k4\i) at ($(k2\i)+(-0.6,0)$) {};
                    \node[vertex] (k5\i) at ($(k3\i)+(-0.6,0)$) {};
                    \draw[edge] (k2\i)--(k5\i) (k3\i)--(k4\i);
                    \draw[rededge] (k2\i)--(k4\i) (k3\i)--(k5\i);
                }
                \foreach \i in {2,4}
                {
                    \draw[rededge] (k1\i)--(k2\i) (k4\i)--(k5\i) (k3\i)--(k1\i);
                    \draw[edge] (u\i)--(k1\i);
                }
        
                \foreach \i in {1,3,5}
                {
                    \draw[edge] (k1\i)--(k2\i) (k4\i)--(k5\i) (k3\i)--(k1\i);
                    \draw[rededge] (u\i)--(k1\i);
                }
                \node[vertex] (p1) at (-0.5,-0.25) {};
                \node[vertex] (p2) at (-0.5,-1.75) {};
                \node[vertex] (p3) at (0.4,-0.6) {};
                \node[vertex] (p4) at (0.4,-1.4) {};
                \draw[rededge] (u2)--(p1);
                \draw[rededge] (p1)--(p3);
                \draw[rededge] (p3)--(p4);
                \draw[rededge] (p4)--(p2);
                \draw[rededge] (p2)--(u2);
        
                \node[vertex] (m) at (-0.4,1) {};
                \draw[rededge] (u4)--(m);
        
                \node[vertex] (t1) at (0.3,1.1) {};
                \node[vertex] (t2) at (0.1, 2.1) {};
                \node[vertex] (t3) at (1, 1.8) {};
                \draw[rededge] (t1)--(t2);
                \draw[rededge] (t2)--(t3);
                \draw[rededge] (t3)--(t1);
        
                \node[vertex] (b1) at (1,0.6) {};
                \node[vertex] (b2) at (1.2,-0.4) {};
                \draw[rededge] (b1)--(b2);
            \end{scope}
        \end{tikzpicture}
        \caption{Red edges in $G$ and $G'$ are the edges of the perfect $v$-factor and perfect $[1,2]$-factor, respectively. The vertex $v$ of $G$ is labelled, and the neighbours of $v$ in $G$ are both labelled in $G$ and $G'$ by $u_1$ to $u_5$.}
        \label{figure:proof_of_RSI_prop}
    \end{figure}

    First, assume that $G'$ has a perfect $[1,2]$-factor $F'$. Let $F_1'$ be a component of $F'$ containing a vertex from $V(G) \setminus N_{G}[v]$. Then $V(F_1') \subseteq V(G) \setminus \{v\}$, as otherwise, $F_1'$ contains an edge of the form $uw_u$ for some $u \in N_G(v)$.
    Then $F_1'$ contains at least two edges and hence is a cycle, contradicting the fact that $uw_u$ is a bridge of $G'$.
    This shows that the connected components of $F'[V(G) \setminus \{v\}]$ in $G$ are either a $K_2$, a cycle graph, or an isolated vertex $u \in N_G(v)$. Let $U = \Iso(F'[V(G) \setminus \{v\}]) \subseteq N_G(v)$, and let $F$ be the disjoint union of the graph $F'[V(G) \setminus \{v\}] - U$ and the star with centre $v$ and the vertices of $U$ as leaves. By the above, $F$ is a $v$-factor of $G$.
    
    Now assume that $G$ has a $v$-factor $F$. Let $F_1' = F[V(G)\setminus N_{F}[v]]$, for $u \in N_F(v)$ let $F_u'$ denote a perfect $[1,2]$-factor of $G'[V(H_u) \cup \{u\}]$, and for $u \in N_G(v) \setminus N_F(v)$ denote by $F_u'$ a perfect $[1,2]$-factor of $H_u$. Then the disjoint union of $F_1'$ and all the $F_u'$ for $u \in N_G(v)$ is a perfect $[1,2]$-factor of $G'$.
    
    By \Cref{theorem:Tutte_condition_for_perfect_1_2-factor}, it remains to prove that $G$ has an RSI-decomposition separating $v$ if and only if there is $S \subseteq V(G')$ with $\iso(G'-S) > |S|$. Given an RSI-decomposition $(R,S,I)$ of $G$ separating $v$, note that $S \subseteq V(G) \setminus \{v\} \subseteq V(G')$ fulfils $\iso(G'-S) \geq \iso(G-S) \geq |I| > |S|$.
    
    For the other direction, let $S \subseteq V(G')$ be minimal with respect to inclusion such that $\iso(G'-S) > |S|$. By construction of $H$, we have $|I\cap (V(H_u)\cup \{u\})| \leq |S\cap V(H_u)|$ for all $u\in N_G(v)$, so $S_1=S\setminus V(H_u)$ also satisfies $\iso(G'-S_1)>|S_1|$, as
        \begin{align*}
            \iso(G'-S_1) \geq \iso(G'-S)-|I\cap (V(H_u)\cup \{u\})| > |S| - |S\cap V(H_u)| = |S_1|.
        \end{align*}
    By minimality of $S$, this means that $S=S_1$ and therefore $S\cap V(H_u) = \emptyset$ and hence also $I\cap (V(H_u)\cup \{u\}) =\emptyset$. This means that $S \subseteq V(G)\setminus \{v\}$ and $I = \iso(G-S) \setminus \{v\}$. Putting things together, we see that $(V(G) \setminus (S \cup I), S, I)$ is an RSI-decomposition of $G$ separating $v$.
\end{proof}
Often it is useful to require an RSI-decomposition to be proper. This is, in fact, no restriction.

\begin{lemma}\label{lemma:existence_of_proper_RSI}
    Let $G$ be a graph, $v$ be a vertex of $G$, and assume that there is an RSI-decomposition of $G$ separating $v$. Then there is a proper RSI-decomposition of $G$ separating $v$.
\end{lemma}
\begin{proof}
    Let $(R,S,I)$ be an RSI-decomposition of $G$ separating $v$, such that $|S|$ is minimal, and assume there was an $s$ with $|N(s) \cap I| \leq 1$. 
    Consider the triple $(R_1,S_1,I_1)$ with $R_1 = R \cup \{s\} \cup (N(s)\cap I)$, $S_1 = S \setminus \{s\}$, and $I_1 = I \setminus (N(s)\cap I)$. 
    It is easy to check that this triple again forms an RSI-decomposition of $G$ separating $v$, but with $|S_1| < |S|$, contradicting the minimality of $|S|$. 
\end{proof}

The form in which we use \Cref{proposition:v-factor_RSI-decomp} most often is as follows.

\begin{corollary}\label{corollary:Staller_wins_then_proper_RSI-decomp}
    Let $G$ be a graph with $o(G) = \mathcal S$, and let $v \in V(G)$. Then there is a proper RSI-decomposition of $G$ separating $v$.
\end{corollary}
\begin{proof}
    Since $o(G) = \mathcal S$, there is no $v$-factor of $G$. So there has to be an RSI-decomposition of $G$ separating $v$ by \Cref{proposition:v-factor_RSI-decomp}. By \Cref{lemma:existence_of_proper_RSI}, there now is a proper RSI-decomposition of $G$ separating $v$.
\end{proof}

As a first application of \Cref{corollary:Staller_wins_then_proper_RSI-decomp}, we prove \Cref{thm:plusone}, which we repeat here for convenience.

\plusone*
\begin{proof}[Proof of \Cref{thm:plusone}]
    By \Cref{proposition:tiny_betas}, we may assume that $d \geq 3$. Let $G$ be a graph on $n \leq 2^{d+1}+1$ vertices of minimum degree $d$, and by contradiction, assume $o(G) = \mathcal S$. If $n < 2^{d+1}$, this is a contradiction to the Erd\H{o}s-Selfridge Theorem, and for $n=2^{d+1}$ the Erd\H{o}s-Selfridge Theorem yields that $G$ must be regular, in which case we get a contradiction to \Cref{corollary:Dominator_wins_on_regular_graphs}. So from now on we can assume that $n = 2^{d+1}+1$. In this case, again by the Erd\H{o}s-Selfridge Theorem, the degree sequence of $G$ has to be 
    \[(\underbrace{d,\dots,d}_{2^{d+1}},D) \text{ or } (\underbrace{d,\dots,d}_{2^{d+1}-1},d+1,d+1),\]
    where $D \geq d$. Let $v$ be a maximum degree vertex of $G$, and let $(R,S,I)$ be an RSI-decomposition of $G$ separating $v$, which exists by \Cref{corollary:Staller_wins_then_proper_RSI-decomp}. Double counting the number of edges between $S$ and $I$ gives
    \begin{align}\label{align:double_counting}
        d|S| + d \leq d|I| \le e(S,I) \leq \sum_{s\in S} \deg(s).
    \end{align}
    In particular, there must be at least one vertex of degree greater than $d$ in $S$. In the first case of $G$ having $2^{d+1}$ vertices of degree $d$, this is impossible, since the only vertex of degree potentially greater than $d$ is $v \in R$. In the second case, again one of the vertices of degree $d+1$ is in $R$, so that by \Cref{align:double_counting} we must have $d = 1$, contradicting $d \geq 3$.
\end{proof}

\section{Low Surplus: RSI-decompositions}\label{section:low_surplus}

In this section, we prove that $o(G) \geq \mathcal N$ for graphs $G$ on at most $21$ vertices and of minimum degree $3$, which are in some sense almost regular. We use the following measure of how regular a graph is. 
\begin{definition}\label{definition:surplus}
    Let $G$ be a graph, let $v\in V(G)$, and let $d$ be an integer.
    We define the \emph{surplus of $G$ with respect to $d$  and the vertex $v$} to be 
    \[\surp{d}{G}{v}=\sum_{x\in V(G)\setminus \{v\}} |\deg(x)-d|.\]
\end{definition}
In most cases $v$ is a maximum degree vertex and $d = 3$ is the minimum degree. In these cases, we simply write $\sur(G) = \surp{d}{G}{v}$ and call it \emph{surplus of $G$}. The surplus only depends on the degree sequence, so we may also speak of the surplus of a degree sequence.

Before stating the main theorem of this section, let us first introduce a convenient reduction telling us that for proving lower bounds for $\beta(3)$, we only have to consider graphs that do not have edges between vertices of degree greater than $3$.

\begin{definition}
    Let $G$ be a graph with $\mindeg(G) = d$. Then $G$ is called \emph{reduced} if $uv \in E(G)$ implies $\deg_G(u) = d$ or $\deg_G(v) = d$.
\end{definition}

\begin{lemma}\label{lemma:reduction_to_reduced_graphs}
    Let $n$ and $d$ be positive integers, and assume that $o(G) \geq \mathcal N$ for all reduced graphs $G$ on $n$ vertices of minimum degree $d$. Then $o(G) \geq \mathcal N$ for all graphs $G$ on $n$ vertices of minimum degree $d$.
\end{lemma}
\begin{proof}
    Let $G$ be a graph on $n$ vertices with $\mindeg(G) = d$. By iteratively deleting edges between vertices of degree greater than $d$, we find a reduced spanning subgraph $G'$ of $G$. The assertion now follows from $o(G) \geq o(G')$.
\end{proof}

What we prove in this section is the following.

\begin{restatable}{theorem}{betathreesureigth}\label{theorem:beta_3_for_sur_8}
    Let $G$ be a reduced graph on at most $21$ vertices with $\mindeg(G) = 3$ and $\sur(G) \leq 8$. Then $o(G) \geq \mathcal N$. 
\end{restatable}

The rough plan for the proof is as follows. Assume that $G$ contradicts \Cref{theorem:beta_3_for_sur_8}, so that by \Cref{corollary:Staller_wins_then_proper_RSI-decomp} there is a proper RSI-decomposition of $G$ separating a maximum degree vertex $v \in V(G)$. Since $\sur(G)\leq 8$, the number of edges between $R$ and $S$ is at most $5$ by the subsequent \Cref{lemma:surplus_in_RSI}. In most cases, deleting the edges in $E(R,S)$ yields two graphs, $G_R$ and $G_{SI}$, one with outcome $\mathcal D$ and the other one with outcome at least $\mathcal N$, contradicting $o(G) = \mathcal S$. The remaining cases are handled similarly. Checking that these two smaller parts have the desired outcomes happens in the subsequent R-lemmas and SI-lemmas, whose proofs heavily rely on brute-force computations.

Given an RSI-decomposition $(R,S,I)$ of a graph $G$ with minimum degree $d$, we often need to bound the number of edges between $R$ and $S$ or the number of vertices of degree less than $d$ in $G[R]$. The following lemma provides us with a useful equality, giving us these desired bounds in terms of the surplus $\surp{d}{G}{v}$.

\begin{lemma}\label{lemma:surplus_in_RSI}
    Let $G$ be a graph with $\mindeg(G) = d$, let $v$ be a vertex of $G$, and let $(R,S,I)$ be an RSI-decomposition of $G$ separating $v$. Then 
    \begin{align}\label{align:surplus_on_RSI}
        \surp{d}{G}{v} = \sum_{r \in R \setminus \{v\}} (\deg(r) - d) + e(R,S) + 2e(S) + d(|I|-|S|) + 2 \sum_{i \in I} (\deg(i) - d).
    \end{align}
\end{lemma}
\begin{proof}
    Note that since $I \subseteq \iso(G-S)$, we have
    \begin{align*}
        \sum_{s \in S} \deg(s) = e(R,S) + 2e(S) + \sum_{i \in I} \deg(i).
    \end{align*}
    Using the definition of surplus, we therefore have
    \begin{align*}
        \surp{d}{G}{v} &=\sum_{r \in R \setminus \{v\}} (\deg(r) - d) + \sum_{s \in S} \deg(s) - d|S| + \sum_{i \in I} (\deg(i) - d)\\
        &=\sum_{r \in R \setminus \{v\}} (\deg(r) - d) + e(R,S) + 2e(S) + d(|I|-|S|) + 2 \sum_{i \in I} (\deg(i) - d).
    \end{align*}
\end{proof}
In fact, we only use the lower bound on the surplus in \Cref{align:surplus_on_RSI}; that is, given an upper bound on the surplus, we use it to bound $e(R,S)$, for example. A very common application of \Cref{lemma:surplus_in_RSI} tells us the following. Given a graph $G$ and a vertex $v \in V(G)$ such that $\mindeg(G) = 3$ and $\sur_3(G,v) \leq 8$, and given an RSI-decomposition $(R,S,I)$ separating $v$, we have $e(R,S) \leq 5$. Note that this follows because each summand on the right of \Cref{align:surplus_on_RSI} is non-negative and because $|I|>|S|$.

The next lemma shows that the surplus of $G_R$ can be bounded in terms of the surplus of $G$, even though a maximum degree vertex in $G_R$ is not necessarily a maximum degree vertex in $G$. 

\begin{lemma}\label{lemma:replace_v_by_w_in_R}
    Let $G$ be a graph with $\mindeg(G) = d$, let $v$ be a maximum degree vertex of $G$, and let $(R,S,I)$ be an RSI-decomposition of $G$ separating $v$. If $w$ denotes a maximum degree vertex in $G_R= G[R]$, then
    \[
        \surp{d}{G_R}{w} \leq \surp{d}{G}{v} - d(|I|-|S|).
    \]
    Furthermore, if equality holds in the above inequality, we must have $\deg_{G_R}(w) = \deg_G(v)$. 
\end{lemma}
\begin{proof}
    We have 
    \begin{align*}
        &\sum_{s \in S} (\deg_G(s)-d) \geq e(R,S)+ d(|I|-|S|).
    \end{align*}
    So by definition of surplus, we get
    \begin{align*}
        \surp{d}{G}{v} &\geq \sum_{\substack{r \in R \setminus \{v\}}} (\deg_{G}(r)-d)+\sum_{s \in S} (\deg_G(s)-d)\\ 
        &\geq \sum_{\substack{r \in R \setminus \{v\} \\ \deg_{G_R}(r) > d}} (\deg_{G_R}(r)-d) + e(R,S) + d(|I|-|S|)\\
        &\geq \left(\sum_{r \in R \setminus \{v\}} |\deg_{G_R}(r)-d| \right) +\deg_G(v)-\deg_{G_{R}}(v) + d(|I|-|S|),
    \end{align*}
    where the third inequality holds because
    \[
        e(R,S) \geq \sum_{\substack{r \in R \setminus \{v\} \\ \deg_{G_R}(r) < d}} |\deg_{G_R}(r)-d| + \deg_G(v)-\deg_{G_{R}}(v).
    \]    
    From this and the fact that 
    \begin{align}\label{align:triangle_ineq_for_replacing_v_with_w}
        |\deg_{G_R}(v) - d| - |\deg_{G_R}(w) - d| \leq \deg_{G_R}(w) - \deg_{G_R}(v) \leq \deg_{G}(v) - \deg_{G_R}(v),
    \end{align}
    we indeed get 
    \begin{align*}
        \sum_{r \in R \setminus \{w\}} |\deg_{G_R}(r) - d| &\leq \left(\sum_{r \in R \setminus \{v\}} |\deg_{G_R}(r)-d| \right) +\deg_G(v)-\deg_{G_{R}}(v) \\
        &\leq \surp{d}{G}{v}-d(|I|-|S|). 
    \end{align*}
    Furthermore, for equality to hold in \Cref{align:triangle_ineq_for_replacing_v_with_w}, we in particular must have $\deg_{G_R}(w) = \deg_G(v)$.
\end{proof}

Let us define the potential of a degree sequence $(d_1,\dots,d_n)$ to be $\sum_{i=1}^{n}2^{-d_i-1}$.
In the proof of many of the subsequent lemmas, we exploit the fact that if the potential of a degree sequence is less than $1$ (respectively $\frac12$), the Dominator wins going first (respectively second) on all graphs with this degree sequence. By bounding the potential after one step into the game, that is, after the Dominator claims a maximum degree vertex, we can find even more such degree sequences for which the Dominator wins on all graphs corresponding to that degree sequence. 

\begin{lemma}[One-step Erd\H os-Selfridge for degree sequences]\label{lemma:one_step_deg_seq}
    Let $G$ be a graph with degree sequence $(d_1,\dots,d_n)$, and assume that 
    \[\sum_{i=1}^{n-d_n-1} 2^{-d_i-1} < \frac 12.\]
    Then $o(G) \geq \mathcal N$. 
\end{lemma}

\begin{proof}
    Let $v$ be a vertex of maximum degree $d_n$. Then 
    \begin{align*}
        \pot(G, \{v\}, \{\}) = \sum_{u \in V(G) \setminus N[v]} 2^{-\deg(u)-1} \leq \sum_{i=1}^{n-d_n-1} 2^{-d_i-1} < \frac 12.
    \end{align*}
    By \Cref{theorem:ES_for_gs} we have $o(G, \{v\},\{\}) = \mathcal D$ and thus $o(G) \geq \mathcal N$. 
\end{proof} 

\subsection*{R-lemmas and SI-lemmas}

As outlined above, using an RSI-decomposition, we split a big graph $G$ on at most $21$ vertices into two smaller parts, $G_R$ and $G_{SI}$. For proving $o(G) \geq \mathcal N$, we then need that the Dominator wins on one of the parts going first and on one of them going second. These desired lower bounds for the outcome of the two parts are proven in the subsequent R-lemmas and SI-lemmas.

\begin{r-lemma}\label{r-lemma:R_4}
    Let $G_{R}$ be a graph on the vertex set $R$ with $|R| = 4$. If $\mindeg(G_R) \geq 1$ and $\sur(G_R) \leq 5$, then $o(G_R) = \mathcal D$. 
\end{r-lemma}

\begin{proof}
    There is exactly one graph on four vertices with no isolated vertex and no perfect $[1,2]$-factor, and this is the unique graph with degree sequence $(1,1,1,3)$. This graph has surplus $6$.
\end{proof}

\begin{r-lemma}\label{r-lemma:R_5_6}
    Let $G_{R}$ be a graph on the vertex set $R$ with $5 \leq |R| \leq 6$, and suppose that $\mindeg(G_R) \geq 1$ and $\sur(G_R) \leq 5$. Then either $o(G_R) = \mathcal D$ or there exists $u \in V(G_R)$ with $\deg_{G_R}(u) = 1$ and $o(G-u) = \mathcal D$.
\end{r-lemma}

\begin{proof}
    In the section on R-Lemmas in \texttt{generate\_and\_check.sh}, we check whether the Dominator, as the second player, wins on every graph on $5$ or $6$ vertices with minimum degree $\geq 1$ and surplus at most $5$. In total \pgfmathprintnumber{145} graphs are checked, and we see that there are $6$ exceptions on which the Staller wins going first. They are listed in \Cref{figure:exceptional_R_graphs}, and it is easy to check that removing any degree $1$ vertex of any one of them results in a graph on which the Dominator wins going second.
\end{proof}

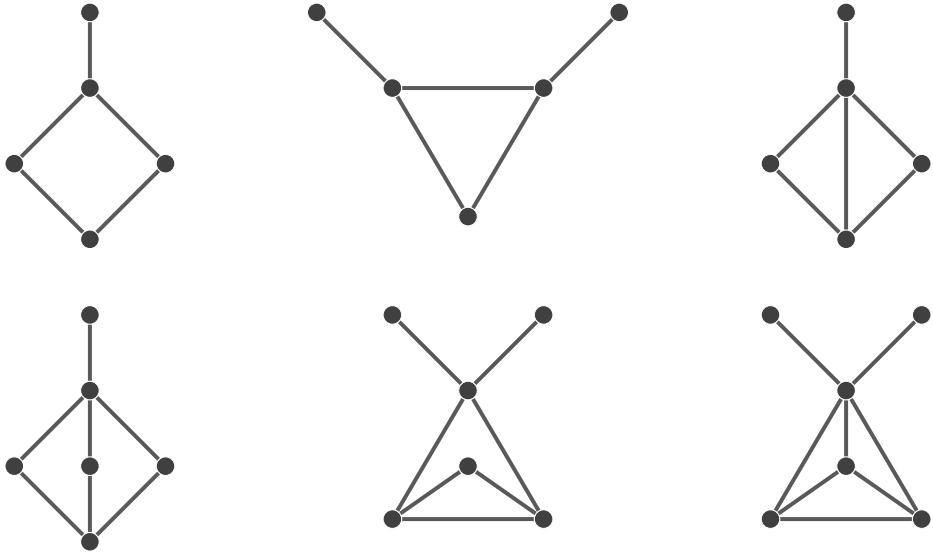
\begin{figure}[ht]
    \centering
    \begin{tikzpicture}[every node/.style={vertex}]
        \begin{scope}[xshift=0cm, yshift=0cm]
            \node (v1) at (-1,-1) {};
            \node (v2) at ( 1,-1) {};
            \node (v3) at ( 0, 1) {};
            \node (v4) at ( 0,-2) {};
            \node (v5) at ( 0, 0) {};
            
            \draw[edge] (v1)--(v4);
            \draw[edge] (v1)--(v5);
            \draw[edge] (v2)--(v4);
            \draw[edge] (v2)--(v5);
            \draw[edge] (v3)--(v5);
        \end{scope}
        
        \begin{scope}[xshift=5cm, yshift=0cm]
            \node (v1) at ( 0,-1.7) {};
            \node (v2) at ( 2, 1.0) {};
            \node (v3) at (-2, 1.0) {};
            \node (v4) at ( 1, 0.0) {};
            \node (v5) at (-1, 0.0) {};
            
            \draw[edge] (v1)--(v4);
            \draw[edge] (v1)--(v5);
            \draw[edge] (v2)--(v4);
            \draw[edge] (v3)--(v5);
            \draw[edge] (v4)--(v5);
        \end{scope}
        
        \begin{scope}[xshift=10cm, yshift=0cm]
            \node (v1) at (-1,-1) {};
            \node (v2) at ( 1,-1) {};
            \node (v3) at ( 0, 1) {};
            \node (v4) at ( 0,-2) {};
            \node (v5) at ( 0, 0) {};
            
            \draw[edge] (v1)--(v4);
            \draw[edge] (v1)--(v5);
            \draw[edge] (v2)--(v4);
            \draw[edge] (v2)--(v5);
            \draw[edge] (v3)--(v5);
            \draw[edge] (v4)--(v5);
        \end{scope}
        
        \begin{scope}[xshift=0cm, yshift=-4cm]
            \node (v1) at (-1,-1) {};
            \node (v2) at ( 1,-1) {};
            \node (v3) at ( 0,-1) {};
            \node (v4) at ( 0, 1) {};
            \node (v5) at ( 0,-2) {};
            \node (v6) at ( 0, 0) {};
            
            \draw[edge] (v1)--(v5);
            \draw[edge] (v1)--(v6);
            \draw[edge] (v2)--(v5);
            \draw[edge] (v2)--(v6);
            \draw[edge] (v3)--(v5);
            \draw[edge] (v3)--(v6);
            \draw[edge] (v4)--(v6);
        \end{scope}
        
        \begin{scope}[xshift=5cm, yshift=-4cm]
            \node (v1) at (-1,-1.7) {};
            \node (v2) at (-1, 1.0) {};
            \node (v3) at ( 1, 1.0) {};
            \node (v4) at ( 1,-1.7) {};
            \node (v5) at ( 0,-1.0) {};
            \node (v6) at ( 0, 0.0) {};
            
            \draw[edge] (v1)--(v4);
            \draw[edge] (v1)--(v5);
            \draw[edge] (v1)--(v6);
            \draw[edge] (v2)--(v6);
            \draw[edge] (v3)--(v6);
            \draw[edge] (v4)--(v5);
            \draw[edge] (v4)--(v6);
        \end{scope}
        
        \begin{scope}[xshift=10cm, yshift=-4cm]
            \node (v1) at (-1,-1.7) {};
            \node (v2) at (-1, 1.0) {};
            \node (v3) at ( 1, 1.0) {};
            \node (v4) at ( 1,-1.7) {};
            \node (v5) at ( 0,-1.0) {};
            \node (v6) at ( 0, 0.0) {};
            
            \draw[edge] (v1)--(v4);
            \draw[edge] (v1)--(v5);
            \draw[edge] (v1)--(v6);
            \draw[edge] (v2)--(v6);
            \draw[edge] (v3)--(v6);
            \draw[edge] (v4)--(v5);
            \draw[edge] (v4)--(v6);
            \draw[edge] (v5)--(v6);
        \end{scope}
    \end{tikzpicture}
    \caption{The six graphs on $5$ or $6$ vertices with no isolated vertex, surplus at most $5$, and outcome at most $\mathcal N$. Their graph6 strings are \gsix{DEw}, \gsix{DEk} and \gsix{DE\{} (top row, left to right)
    and \gsix{E?zo}, \gsix{ECfo} and \gsix{ECfw} (bottom row, left to right). Here, graph6 refers to the string format used to store graphs in \nauty.}
    \label{figure:exceptional_R_graphs}
\end{figure}

\begin{r-lemma}\label{r-lemma:R_7_i=s+2}
    Let $G_{R}$ be a graph on the vertex set $R$ with $4 \leq |R| \leq 7$. If $\mindeg(G_R) \geq 1$ and $\sur(G_R) \leq 2$, then $o(G_R) = \mathcal D$. 
\end{r-lemma}
\begin{proof}
    This is checked in \texttt{generate\_and\_check.sh}. In total \pgfmathprintnumber{1040} graphs are checked and are shown to be won by the Dominator.
\end{proof}

\begin{r-lemma}\label{r-lemma:R_12}
    Let $G_{R}$ be a graph on the vertex set $R$ with $|R| \leq 12$. If $\mindeg(G_R) \geq 1$ and $\sur(G_R) \leq 5$, then $o(G_R) \geq \mathcal N$. 
\end{r-lemma}
\begin{proof}
    We only need to consider graphs with degree sequences $(d_1,\dots,d_n)$ with $d_i\leq d_{i+1}$ for which
    \begin{itemize}
        \item $n \leq 12$ and $d_1 \geq 1$,
        \item the potential of $(d_1,\dots,d_n)$ is at least $1$,
        \item the quantity $\sum_{i=1}^{n-d_n-1} 2^{-d_i-1}$ in \Cref{lemma:one_step_deg_seq} is at least $\frac 12$, 
        \item the surplus $\sum_{i = 1}^{n-1} |d_i-3|$ is at most $5$ and 
        \item the sum of the entries $\sum_{i=1}^n d_i$ is even so that $(d_1,\dots,d_n)$ has a chance of being a degree sequence of a graph by the handshaking lemma.
    \end{itemize}
    By going computationally or by hand through all degree sequences, we obtain the relevant degree sequences listed in \Cref{table:R_12}.
    Note that due to the condition on the surplus, there are at most 6 vertices with degree $\neq3$.
    
    In \Cref{table:R_12} we have $x_1\in\{4,6\}$, $x_2\in\{4,6,8\}$, $y_1\in\{3,5,7\}$ and $y_2\in\{3,5,7,9\}$.
    
    In the section on R-lemmas in \texttt{generate\_and\_check.sh}, we check all graphs having one of these degree sequences.
    Due to the way \nauty\ or in particular \geng\ and \pickg\ work, we actually generate all graph with respective number of vertices and minimum and maximum degree.
    In total, we check \pgfmathprintnumber{5026665} graphs which needs around 5 minutes.
\end{proof}

\begin{r-lemma}\label{r-lemma:R_13_i=s+2}
    Let $G_{R}$ be a graph on the vertex set $R$ with $|R| \leq 13$. If $\mindeg(G_R) \geq 1$ and $\sur(G_R) \leq 2$, then $o(G_R) \geq \mathcal N$.
\end{r-lemma}
\begin{proof}
    Because of the Erd\H{o}s-Selfridge Theorem, a graph contradicting this lemma would need to have the degree sequence $(1,\underbrace{3,\dots,3}_{12})$. There are no graphs with this degree sequence because of the handshaking lemma.
\end{proof}

\begin{table}[!ht]
    \centering
    \rowcolors{2}{white}{black!5!white}
    \begin{tabular}{c  c  c  c}
        \toprule
        length & all possible degree sequences & \makecell{maximum\\degree} & \makecell{number\\of edges}\\
        \midrule
        $9$ & $(1,1,2,3,\dots,3)$ & $3$ & $11$ \\ 
        $10$ & $(1,1,2,3,\dots,3,x_1), (1,1,3,\dots,3)$ & $6$ & $13 - 14$ \\
        $11$ & \makecell{$(1,1,2,3,\dots,3,y_1),(1,2,2,2,3,\dots,3,y_1)$\\
        $(2,2,2,2,2,3,\dots,3),(1,1,3,\dots,3,x_2)$} & $8$ & $14-17$ \\ 
        $12$ & \makecell{$(1,1,2,3,\dots,3,x_2), (1,2,2,2,3,\dots,3,x_2),$ \\
        $(2,2,2,2,2,3,\dots,3,x_1),(1,1,3,\dots,3,y_2),$\\
        $(1,1,3,\dots,3,4,x_2),(1,2,2,3,\dots,3,y_2),$\\
        $(1,2,2,3,\dots,3,4,4),(2,2,2,2,3,\dots,3)$} & $9$ & $16-19$ \\
        \bottomrule
    \end{tabular}
    \caption{Degree sequences to consider for \Cref{r-lemma:R_12}}\label{table:R_12}
\end{table}

\begin{r-lemma}\label{r-lemma:R_14_no_deg_1}
    Let $G_{R}$ be a graph on the vertex set $R$ with $|R| \leq 14$. If $\mindeg(G_R) \geq 2$ and $\sur(G_R) \leq 5$, then $o(G_R) \geq \mathcal N$.
\end{r-lemma}
\begin{proof}
    Again, we only have to consider degree sequences $(d_1,\dots,d_n)$ with $d_i\leq d_{i+1}$ and
    \begin{itemize}
        \item $n \leq 14$ and $d_1 \geq 2$,
        \item the potential of $(d_1,\dots,d_n)$ is at least $1$,
        \item the quantity $\sum_{i=1}^{n-d_n-1} 2^{-d_i-1}$ in \Cref{lemma:one_step_deg_seq} is at least $\frac 12$, 
        \item the surplus $\sum_{i = 1}^{n-1} |d_i-3|$ is at most $5$ and 
        \item the sum of entries $\sum_{i=1}^n d_i$ is even.
    \end{itemize}
    Graphs with $|R|\leq12$ are already covered in \Cref{r-lemma:R_4}, hence we can limit to graphs with 13 and 14 vertices here.
    For these we find the degree sequences from \Cref{table:R_14_no_deg_1} to be checked.
    \begin{table}[!ht]
        \centering
        \rowcolors{2}{white}{black!5!white}
        \begin{tabular}{c c c c} 
            \toprule
            length & all possible degree sequences & \makecell{maximum\\degree} & \makecell{number\\of edges}\\
            \midrule
            $13$ & \makecell{$(2,2,2,2,2,3,\dots,3,y_1), (2,2,2,2,3,\dots,3,x),$\\
            $(2,2,2,3,\dots,3)$} & $8$ & $17-19$ \\ 
            $14$ & \makecell{$(2,2,2,2,2,3,\dots,3,x), (2,2,2,2,3,\dots,3,y_2),$\\
            $(2,2,2,2,3,\dots,3,4,x), (2,2,2,3,\dots,3,x)$\\
            $(2,2,3,\dots,3)$} & $9$ & $19-22$ \\
            \bottomrule 
        \end{tabular}
        \caption{Degree sequences to consider for \Cref{r-lemma:R_14_no_deg_1}}\label{table:R_14_no_deg_1}
    \end{table}
    Here $x\in\{4,6,8\}$, $y_1\in\{3,5,7\}$ and $y_2\in\{3,5,7,9\}$. In \texttt{generate\_and\_check.sh}, all graphs with respective parameters, in total \pgfmathprintnumber{5695198}, are checked, which takes around 105 minutes.
\end{proof}

\begin{si-lemma}\label{si-lemma:K34_special_strat}
    Let $G_{SI}=K_{3,4}$, and let $u_1$ and $u_2$ be two distinct degree $4$ vertices of $G_{SI}$. As the second player, the Dominator has a winning strategy that simultaneously ensures he claims either $u_1$ or $u_2$.
\end{si-lemma}
\begin{proof}
    Note that a set of vertices of $G_{SI}$ containing a vertex of degree $3$ and a vertex of degree $4$ already forms a dominating set. As the second player, the Dominator can guarantee to claim either $u_1$ or $u_2$ in his first move and to claim a degree $3$ vertex in his second move. This proves the lemma.
\end{proof}

\begin{si-lemma}\label{si-lemma:K34_with_extra_vtx}
    Let $G_{SI} \in \{G_{SI}^{(1)}, G_{SI}^{(2)}, G_{SI}^{(1)} - w^{(1)}, G_{SI}^{(2)} - w^{(2)}\}$, where the graphs $G_{SI}^{(1)}$ and $G_{SI}^{(2)}$ and their vertices $w^{(1)}$ and $w^{(2)}$ are depicted in \Cref{figure:K34_with_extra_vtx}. Then $o(G_{SI}) = \mathcal D$. 
\end{si-lemma}

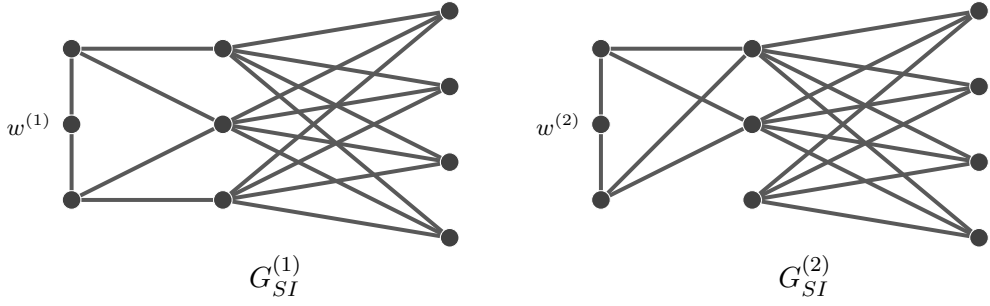
\begin{figure}[ht]
    \centering
    \begin{tikzpicture}
        \begin{scope}
            \foreach \i in {1,2,3} {
                \node[vertex] (L\i) at (0,\i+0.5) {};
            }
            \foreach \j in {1,2,3,4} {
                \node[vertex] (R\j) at (3,\j) {};
            }
            \foreach \i in {1,2,3} {
                \foreach \j in {1,2,3,4} {
                    \draw[edge] (L\i) -- (R\j);
                }
            }
            
            \node[vertex] (u) at (-2, 1.5) {};  
            \node[vertex, label={west:\footnotesize $w^{(1)}$}] (w) at (-2, 2.5) {};
            \node[vertex] (u') at (-2, 3.5) {};
            
            \draw[edge] (u) -- (L1);
            \draw[edge] (u) -- (L2);
            \draw[edge] (u') -- (L2);
            \draw[edge] (u') -- (L3);
            \draw[edge] (u) -- (w);
            \draw[edge] (u') -- (w);
    
            \node at (0.7,0.5) {$G_{SI}^{(1)}$};
        \end{scope}
    
        \begin{scope}[xshift=7cm]
            \foreach \i in {1,2,3} {
                \node[vertex] (L\i) at (0,\i+0.5) {};
            }
            \foreach \j in {1,2,3,4} {
                \node[vertex] (R\j) at (3,\j) {};
            }
            \foreach \i in {1,2,3} {
                \foreach \j in {1,2,3,4} {
                    \draw[edge] (L\i) -- (R\j);
                }
            }
            
            \node[vertex] (u) at (-2, 1.5) {};  
            \node[vertex, label={west:\footnotesize $w^{(2)}$}] (w) at (-2, 2.5) {};
            \node[vertex] (u') at (-2, 3.5) {};
            
            \draw[edge] (u) -- (L3);
            \draw[edge] (u) -- (L2);
            \draw[edge] (u') -- (L2);
            \draw[edge] (u') -- (L3);
            \draw[edge] (u) -- (w);
            \draw[edge] (u') -- (w);
            
            \node at (0.7,0.5) {$G_{SI}^{(2)}$};
        \end{scope}
    \end{tikzpicture}
    \caption{The graphs $G_{SI}^{(1)}$ and $G_{SI}^{(2)}$}
    \label{figure:K34_with_extra_vtx}
\end{figure}

\begin{proof}
    The graphs $G_{SI}^{(1)}$ and $G_{SI}^{(2)}$ have spanning subgraphs having two components isomorphic to $K_2$ and one component isomorphic to $K_{2,4}$. Since $o(K_2) = o(K_{2,4}) = \mathcal D$, we get $o(G_{SI}^{(1)}) = o(G_{SI}^{(2)}) = \mathcal D$. For $G_{SI}^{(2)} - w^{(2)}$, we just apply \Cref{si-lemma:K34_special_strat} to the copy of $K_{3,4}$ contained in $G_{SI}^{(2)} - w^{(2)}$ and to $u_1$ and $u_2$ being the two degree $6$ vertices of $G_{SI}^{(2)} - w^{(2)}$. Finally, $G_{SI}^{(1)} - w^{(1)}$ has a spanning subgraph with one $K_2$ and one $K_{2,5}$ as its components. Since $o(K_2) = o(K_{2,5}) = \mathcal D$, we get $o(G_{SI}^{(1)} - w^{(1)}) = \mathcal D$.
\end{proof}

\begin{si-lemma}\label{si-lemma:S6I7}
    Let $G_{SI}=(S,I,E)$ be a bipartite graph with $|I| = |S|+1\leq 7$, $\deg(i)=3$ for all $i\in I$, and $\deg(s) \geq 2$ for all $s\in S$. Then $o(G_{SI}) = \mathcal D$. 
\end{si-lemma}
\begin{proof}
    This is checked in the section on SI-lemmas in \texttt{generate\_and\_check.sh}.
    In total \pgfmathprintnumber{711} graphs are checked to be won by the Dominator.
\end{proof}

\begin{si-lemma}\label{si-lemma:S7I9}
    Let $G_{SI}=(S,I,E)$ be a bipartite graph with $|I| = |S|+2\leq 9$, $\deg(s) \geq 2$ for all $s\in S$, and at most three vertices of degree $2$ in $S$. Suppose furthermore that there is exactly one vertex in $I$ with degree $2$, and for all other $i \in I$, $\deg(i)=3$. Then $o(G_{SI})\geq \mathcal N$. 
\end{si-lemma}
\begin{proof}
    This is checked in \texttt{generate\_and\_check.sh}.
    In total \pgfmathprintnumber{293287} graphs are checked to be won by the Dominator as starting player, which takes around a minute.
\end{proof}

\begin{si-lemma}\label{si-lemma:S8I9}
    Let $G_{SI}=(S,I,E)$ be a bipartite graph with $|I| = |S|+ 1 \leq 9$, $\deg(i)=3$ for all $i\in I$, and $\deg(s) \geq 2$ for all $s\in S$. Then $o(G_{SI})\geq  \mathcal N$. 
\end{si-lemma}
\begin{proof}
    By \Cref{si-lemma:S6I7}, we only have to consider the cases $8 \leq |I| = |S|+1 \leq 9$. All the remaining graphs are checked in \texttt{generate\_and\_check.sh}.
    In total \pgfmathprintnumber{274425} graphs are checked to be won by the Dominator as starting player, which takes around three minutes.
\end{proof}

\begin{si-lemma}\label{si-lemma:S9I10}
    Let $G_{SI}=(S,I,E)$ be a bipartite graph with $|I| = |S|+ 1 \leq 10$, $\deg(i)=3$ for all $i\in I$ and $2 \leq \deg(s) \leq 4$ for all $s\in S$. Then $o(G_{SI})\geq  \mathcal N$. 
\end{si-lemma}
\begin{proof}
    By \Cref{si-lemma:S8I9}, we only have to consider the case $|I| = |S|+ 1 = 10$.
    In this case we have $\sum_{s\in S} \deg(s) = \sum_{i\in I} \deg(i) = 30$. If there were four vertices of degree 2,  then even if all other vertices had degree 4, we would get $\sum_{s\in S} \deg(s) = 4\cdot 2 + 5\cdot 4 < 30$. Hence, there are at most three vertices of degree 2.
    We check all such graphs in \texttt{generate\_and\_check.sh}.
    In total \pgfmathprintnumber{465023} graphs are checked to be won by the Dominator as starting player, which takes around 46 minutes.
\end{proof}

\begin{si-lemma}\label{si-lemma:S5I7_i=s+2}
    Let $G_{SI}=(S,I,E)$ be a bipartite graph with $|I| = |S|+ 2 \leq 7$, $\deg(i)=3$ for all $i\in I$ and  $\deg(s) \geq 2$ for all $s\in S$ Then $o(G_{SI})= \mathcal D$. 
\end{si-lemma}
\begin{proof}
    Note that $|S| \geq 3$, since the vertices in $I$ have degree $3$. In this case we have $\sum_{s\in S} \deg(s) = \sum_{i\in I} \deg(i) = 3\cdot|I|$. If there were three vertices of degree 2, then even if all other vertices had degree 7, we would get $\sum_{s\in S} \deg(s) = 3\cdot 2 + (|S|-3)\cdot 7 =  6+ (|I|-5)\cdot 7 < 3\cdot|I|$ for $|I|\leq 7$. Hence, there are at most two vertices of degree 2. We check all such graphs with $5 \leq |I| = |S|+ 2 \leq 7$ in \texttt{generate\_and\_check.sh}.
    In total \pgfmathprintnumber{117} graphs are checked to be won by the Dominator.
\end{proof}

\begin{si-lemma}\label{si-lemma:S7I9_i=s+2}
    Let $G_{SI}=(S,I,E)$ be a bipartite graph with $|I| = |S|+ 2 \leq 9$, $\deg(i)=3$ for all $i\in I$, $\deg(s) \geq 2$ for all $s\in S$, and with at most two vertices of degree $2$. Then $o(G_{SI})\geq \mathcal N$. 
\end{si-lemma}
\begin{proof}
    Because of the previous \Cref{si-lemma:S5I7_i=s+2}, we can assume that $8 \leq |I| = |S|+ 2 \leq 9$. We again check all such graphs in \texttt{generate\_and\_check.sh}.
    In total \pgfmathprintnumber{74183} graphs are checked to be won by the Dominator as starting player.
\end{proof}

\subsection*{RSI-decomposition and Proof of \Cref{theorem:beta_3_for_sur_8}}
It is convenient to deal with special cases of the RSI-decomposition that we get from \Cref{corollary:Staller_wins_then_proper_RSI-decomp}. Let us first deal with the case $|I| \geq |S|+2$, such that later we can assume $|I| = |S|+1$.

\begin{lemma}\label{lemma:beta_3_disconnected}
    Let $G$ be a graph on at most $21$ vertices and $\mindeg(G) = 3$, and assume that $G$ is not connected. Then $o(G) \geq \mathcal N$.
\end{lemma}
\begin{proof}
    Let $V_1$ and $V_2$ be a partition of $V(G)$ such that $0 < |V_1| \leq |V_2|$ and such that $e(V_1,V_2) = 0$. Since $\mindeg(G) = 3$, we must have $|V_1|\geq 4$ and hence $|V_1| \leq 10$ and $|V_2| \leq 17$. Since $\beta'(3) = 13$ by \Cref{proposition:tiny_betas} and $\beta(3) \geq 18$ by \Cref{thm:plusone}, the Dominator thus wins on $G[V_1]$ going second and on $G[V_2]$ going first. Altogether $o(G) \geq \mathcal N$.
\end{proof}

\begin{proposition}\label{proposition:i=s+2}
    Let $G$ be a graph on at most $21$ vertices with $\mindeg(G) = 3$ and $\sur(G) \leq 8$, let $v$ be a maximum degree vertex of $G$, and assume there is a proper RSI-decomposition $(R,S,I)$ of $G$ separating $v$ such that $|I| \geq |S|+2$. Then $o(G) \geq \mathcal N$.  
\end{proposition}
\begin{proof}
    By \Cref{lemma:surplus_in_RSI}, we have $|I| = |S|+2$ and can assume that $\deg(i) = 3$ for all $i \in I$ and that $e(S) = 0$, as otherwise $G$ would not be connected, in which case the Dominator wins by \Cref{lemma:beta_3_disconnected}.
    
    Let $G_R = G[R]$ and let $G_{SI} = G[S \cup I]$. We collect some properties of $G_R$ and $G_{SI}$ to be able to apply the lemmas from the previous section. First, $|R| \geq 4$, as otherwise $e(R,S) \geq 3$, contradicting \Cref{lemma:surplus_in_RSI}. There is at least one vertex in $I$, which has to have at least three neighbours in $S$, so that $|I|= |S|+2 \geq 5$. Altogether, we get that $4 \leq |R| \leq 13$ and $5 \leq |I|= |S|+2 \leq 9$. Since $(R,S,I)$ is proper, all the vertices $s \in S$ have degree at least $2$ in $G_{SI}$, and by \Cref{lemma:surplus_in_RSI} there are at most two vertices of degree $2$ in $S$. By definition of RSI-decomposition, we also have $\deg_{G_{SI}}(i) = \deg_G(i) = 3$ for all $i\in I$. For $G_R$, note that \Cref{lemma:replace_v_by_w_in_R} implies $\sur(G_R) \le 2$. By \Cref{lemma:surplus_in_RSI}, we have $e(R,S) \leq 2$ so that $\deg_{G_R}(r) \geq 1$ for all $r \in R$.
    
    Depending on the size of $R$, we now distinguish two cases. 

    \textbf{Case 1. } Suppose that $4 \leq |R| \leq 7$. By the above, $G_R$ satisfies all the conditions in \Cref{r-lemma:R_7_i=s+2}, so that $o(G_R) = \mathcal D$. Likewise, $G_{SI}$ fulfils all the conditions on \Cref{si-lemma:S7I9_i=s+2} so that $o(G_{SI}) \geq \mathcal N$ and thus $o(G) \geq \mathcal N$.

    \textbf{Case 2. } Suppose that $8 \leq |R| \leq 13$. Again, $G_R$ satisfies the conditions for \Cref{r-lemma:R_13_i=s+2}, giving $o(G_R) \geq N$. For this range of $|R|$, we have that $|I| = |S|+2 \leq 7$, so that by \Cref{si-lemma:S5I7_i=s+2} $o(G_{SI}) = \mathcal D$. Again $o(G) \geq \mathcal N$. 
\end{proof}

Another convenient assumption for later is that $G_R = G[R]$ has no isolated vertices. For vertices $r\in R\setminus\{v\}$, this follows from the definition of RSI-decomposition. We now deal with the case that all the neighbours of $v$ are in $S$.

\begin{proposition}\label{proposition:mindeg_0_in_GR}
    Let $G$ be a reduced graph on at most $21$ vertices with $\mindeg(G) = 3$ and $\sur(G) \leq 8$, and assume that for every maximum degree vertex $v$ of $G$ there is an RSI-decomposition $(R,S,I)$ of $G$ separating $v$ with $N(v) \subseteq S$. Then $o(G) \geq \mathcal N$.
\end{proposition}
\begin{proof}
    Assume there exists a reduced graph $G$ contradicting the proposition. Let $v$ be an arbitrary maximum degree vertex of $G$, and let $(R,S,I)$ be an RSI-decomposition of $G$ separating $v$ such that $N(v) \subseteq S$. Note that because of \Cref{lemma:surplus_in_RSI}, we have $e(R,S)\leq 5$, so that the maximum degree of $G$ is either $3,4$ or $5$, since all the edges containing $v$ are between $R$ and $S$. The Dominator wins on $3$-regular graphs by \Cref{corollary:Dominator_wins_on_regular_graphs}, so we only have to consider the maximum degree to be $4$ or $5$. Furthermore, note that \Cref{lemma:surplus_in_RSI} implies that $|I| = |S| + 1$, since $e(R,S) \geq 4$.

     \textbf{Case 1. } Suppose that there is a maximum degree vertex $v$ of $G$ and an RSI-decomposition $(R,S,I)$ of $G$ separating $v$ such that $N(v) \subseteq S$ and $|R| \geq 2$. If $G$ is disconnected, we are finished by \Cref{lemma:beta_3_disconnected}, so there is some $u \in R\setminus \{v\}$ with a neighbour in $S$. \Cref{lemma:surplus_in_RSI} already forces $\deg(v) = 4$ and all $r \in R \setminus \{v\}$ to have degree $3$ in $G$. Furthermore, the four edges containing $v$ and the one edge containing $u$ and its neighbour in $S$ are the only edges in $E(R,S)$. The connected component $G_u$ of $G[R]$ in which $u$ lies thus has only degree $3$ vertices except for $u$, which has degree $2$ in $G_u$. Such a graph $G_u$ has at least $5$ vertices, so that $|R| \geq 6$ and hence $|I| = |S|+1 \leq 8$. Since $G$ is reduced, the $4$ neighbours of $v$ have at most two neighbours in $I$, and the other $|S| - 4$ vertices of $S$ have degree at most $4$, so that the number of edges between $S$ and $I$ is bounded by 
    \[
        e(S,I) \leq 4(|S|-4) + 8 = 4|I| - 12 < 3|I|,
    \]
    where the last inequality comes from $|I| \leq 8$. This is impossible.

    \textbf{Case 2. } Suppose that for every maximum degree vertex $v$ of $G$ there is an RSI-decomposition $(R,S,I)$ of $G$ separating $v$ such that $N(v) \subseteq S$ and $R = \{v\}$. 

    \textbf{Case 2a. } Suppose that $G$ has maximum degree $5$. We show that the degree sequence of $G$ is
    \[
        (\underbrace{3,\dots,3}_{15},5,5,5,5,5)
    \]
    and that the neighbourhoods of any two degree $5$ vertices are disjoint, which is clearly impossible.

    Fix a degree $5$ vertex $v$ and an RSI-decomposition $(R,S,I)$ separating $v$. By assumptions, we have $R = \{v\}$, and by \Cref{lemma:surplus_in_RSI} all the vertices in $I$ have degree $3$ and $e(S) = 0$. We look at the number of edges between $S$ and $I$. The five neighbours of $v$ all have at most two neighbours in $I$, since they have degree $3$ in the reduced graph $G$. The other vertices in $S$ have at most $5$ neighbours in $I$. We can thus bound the number of edges between $S$ and $I$ by
    \[
        3|I| =  e(S,I) \leq 5(|S|-5) + 10 = 5|I| - 20,
    \]
    with the inequality being strict when any of the vertices in $S \setminus N(v)$ has degree less than $5$. Because $|I| \leq 10$, this thus forces $|I| = 10$ and all the vertices in $S \setminus N(v)$ to have degree exactly $5$. The five neighbours of $v$ and the vertices in $I$ have degree exactly $3$, so that $G$ indeed has the claimed degree sequence. Note further that for any degree $5$ vertex $u\in S\setminus N(v)$, we have $N(u) \subseteq I$, so that $N(v) \cap N(u) = \emptyset$. Since $v$ was arbitrary, the neighbourhoods of any two degree $5$ vertices must be disjoint, which proves the claim.

    \textbf{Case 2b. } Suppose that $G$ has maximum degree $4$. Fix any degree $4$ vertex $v$ and RSI-decomposition $(R,S,I)$ of $G$ separating $v$, for which, by assumptions, $R = \{v\}$. The neighbours of $v$ have at most two neighbours in $I$, and there are at most $|S|-4$ vertices of degree $4$ in $S$, so that the number of edges between $S$ and $I$ is at most
    \[3|I| \leq e(S,I) \leq 4(|S|-4) + 8 = 4|I|-12 < 3|I|,\]
    where the last inequality comes from $|I| \leq 10$. This is a contradiction.
\end{proof}

\begin{proposition}\label{proposition:preprocessed_RSI}
    Let $G$ be a reduced graph on at most $21$ vertices with $\mindeg(G) = 3$ and $\sur(G) \leq 8$ such that $o(G) = \mathcal S$. Then there exists a maximum degree vertex $v$ of $G$, a proper RSI-decomposition $(R,S,I)$ of $G$ separating $v$ and a subgraph $G_{SI}$ of $G$ on the vertex set $S \cup I$ such that for $G_R = G[R]$ the following holds.
    \begin{enumerate}
        \item\label{item:preprocessed_RSI:vert} 
            $2 \leq |R| \leq 14$ and $4 \leq |I| = |S| + 1 \leq 10$.
        \item\label{item:preprocessed_RSI:degGR} 
            $\deg_{G_{R}}(r) \geq 1$ for all $r \in R$.
        \item\label{item:preprocessed_RSI:sur_in_GR}
            $\sur(G_R) \leq 5$.
        \item\label{item:preprocessed_RSI:bipartite} 
            $G_{SI}$ is bipartite. 
        \item\label{item:preprocessed_RSI:degGSI} 
            $\deg_{G_{SI}}(s) \geq 2$ for all $s \in S$ and $\deg_{G_{SI}}(i) = 3$ for all $i \in I$.
    \end{enumerate}
\end{proposition}
\begin{proof}
    By \Cref{corollary:Staller_wins_then_proper_RSI-decomp}, there exists a proper RSI-decomposition $(R,S,I)$ separating $v$ for every maximum degree vertex $v$ of $G$. If $N(v) \subseteq S$ for every such $v$, we have $o(G) \geq \mathcal N$ by \Cref{proposition:mindeg_0_in_GR}, contradicting the assumptions, so we may assume that there is some maximum degree vertex $v$ and a proper RSI-decomposition $(R,S,I)$ separating $v$, such that $N(v) \cap R \neq \emptyset$. We prove the proposition for this choice of $v$ and $(R,S,I)$. 

    Let us first prove \ref{item:preprocessed_RSI:vert}. If $|I| \geq |S| + 2$, \Cref{proposition:i=s+2} implies $o(G) \geq \mathcal  N$, contradicting $o(G) = \mathcal S$, so that $|I| = |S|+1$. Note that $v$ is not the only vertex in $R$, since otherwise $N(v) \subseteq S$. Furthermore, note that $S$ contains at least $3$ neighbours of some vertex in $I \neq \emptyset$ so that $|I| = |S|+1 \geq 4$. Finally, the upper bounds on the sizes of $R,S$ and $I$ come from $G$ having at most $21$ vertices. 
    
    We already know that $v$ has a neighbour in $R$, and by definition of RSI-decompositions, it must be that $r \notin \Iso(G-S)$ for every $r \in R \setminus \{v\}$, so that \ref{item:preprocessed_RSI:degGR} follows. \Cref{lemma:replace_v_by_w_in_R} immediately gives \ref{item:preprocessed_RSI:sur_in_GR}.

    For the definition of $G_{SI}$ and the properties \ref{item:preprocessed_RSI:bipartite}--\ref{item:preprocessed_RSI:degGSI}, we consider three cases.

    \textbf{Case 1.} Suppose that $\deg_G(i) = 3$ for all $i \in I$. In this case, we define $G_{SI} = (S \cup I, E(S,I))$. Clearly $G_{SI}$ is bipartite, so \ref{item:preprocessed_RSI:bipartite} follows. Property \ref{item:preprocessed_RSI:degGSI} holds because $(R,S,I)$ is proper.
    
    \textbf{Case 2.} There is exactly one element $i_1 \in I$ with $\deg_G(i_1) \geq 4$. In particular, this vertex has degree $4$ or $5$, as otherwise, we have a contradiction to \Cref{lemma:surplus_in_RSI}. Consider the graph $G_1 = (S \cup I, E(S,I))$. By \Cref{lemma:surplus_in_RSI}, we have that $e(R,S)+2e(S) \leq 3$, so that there are at most three vertices of degree $2$ in $G_1$ while all the other vertices have degree at least $3$. We therefore find a set of vertices $S_1 \subseteq N_{G_1}(i_1) = N_{G}(i_1) \subseteq S$ with $|S_1| = \deg_G(i_1)-3$, all of degree at least $3$ in $G_1$. By construction, $G_{SI} = G_1 - \{si_1:s\in S_1\}$ is bipartite and all $s\in S$ have degree at least $2$ in $G_{SI}$. By assumption, $i_1$ was the only vertex of degree greater than $3$ in $I$, so that $\deg_{G_{SI}}(i) = 3$ for all $i\in I$. 
    
    \textbf{Case 3.} There are at least two elements $i_1,i_2 \in I$ with $\deg_G(i_1), \deg_G(i_2) \geq 4$. Again, consider the graph $G_1 = (S \cup I, E(S,I))$. From \Cref{lemma:surplus_in_RSI} we get that $\deg_G(i_1)= \deg_G(i_2) = 4$, that there does not exist another $i_3 \in I \setminus \{i_1,i_2\}$ with $\deg(i_3) \geq 4$ and that there is at most one vertex in $S$ of degree $2$ in $G_1$. We therefore find two distinct vertices $s_1 \in N_{G_1}(i_1) = N_{G}(i_1) \subseteq S$ and $s_2 \in N_{G_1}(i_2) = N_{G}(i_2) \subseteq S$ both of degree at least $3$ in $G_1$. By construction, $G_{SI} = G_1 - \{s_1i_1,s_2i_2\}$ is bipartite and all $s\in S$ have degree at least $2$ in $G_{SI}$. We already know that $i_1$ and $i_2$ were the only vertices of degree greater than $3$ in $I$, so that $\deg_{G_{SI}}(i) = 3$ for all $i\in I$.
\end{proof}

We are now ready to prove \Cref{theorem:beta_3_for_sur_8}, which we restate for convenience.
\betathreesureigth*
\begin{proof}
    Assume the statement does not hold, that is there is a reduced graph $G$ on at most $21$ vertices, with minimum degree $3$, $\sur(G) \leq 8$ and $o(G) = \mathcal S$. By \Cref{proposition:preprocessed_RSI} there is a maximum degree vertex $v$ of $G$, an RSI-decomposition $(R,S,I)$ of $G$ separating $v$ and graphs $G_{SI}$ and $G_R = G[R]$ such that conditions \ref{item:preprocessed_RSI:vert}--\ref{item:preprocessed_RSI:degGSI} are satisfied. Note that $e(R,S) \leq 5$ by \Cref{lemma:surplus_in_RSI}. Depending on the size of $R$, we consider the following cases.

    \textbf{Case 1.} Suppose that $|R| = 2$. By \ref{item:preprocessed_RSI:degGR} in \Cref{proposition:preprocessed_RSI}, $G_R$ is isomorphic to $K_2$, on which the Dominator wins going second. Note that $|S|+1 = |I| \leq 10$ and that $\deg_G(v) \leq 4$, because $e(R,S) \leq 5$ and since all the vertices in $R$ must have degree at least $3$ in $G$. It follows that every vertex in $G_{SI}$ has degree at most $4$, thus we can apply \Cref{si-lemma:S9I10}. We get $o(G_{SI}) \geq \mathcal N$, and thus $o(G) \geq \mathcal N$.
    
    \textbf{Case 2.} Suppose that $|R| = 3$. There are only two graphs on three vertices fulfilling \ref{item:preprocessed_RSI:degGR} in \Cref{proposition:preprocessed_RSI}, namely $P_3$ and $K_3$. If $G_R$ is isomorphic to $P_3$, then $\deg_G(v) = 3$ since $e(R,S) \leq 5$ and since all the vertices in $R$ have degree at least $3$ in $G$. In this case, $G$ is regular and $o(G) \geq \mathcal N$ by \Cref{corollary:Dominator_wins_on_regular_graphs}. For $G_R$ being isomorphic to $K_3$, note that $|S|+1 = |I| \leq 9$ so that from \Cref{si-lemma:S8I9} it follows that $o(G_{SI}) \geq \mathcal N$. Since $o(K_3) = \mathcal D$, we get $o(G) \geq \mathcal N$.

    \textbf{Case 3.} Suppose that $|R| = 4$. \Cref{r-lemma:R_4} implies that $o(G_R) = \mathcal D$. Note that $|S|+1 = |I| \leq 9$ so that from \Cref{si-lemma:S8I9} it follows that $o(G_{SI}) \geq \mathcal N$ and therefore $o(G) \geq \mathcal N$.

    \textbf{Case 4.} Suppose that $5 \leq |R| \leq 6$. We apply \Cref{r-lemma:R_5_6} and split into two further subcases depending on the conclusion in this lemma. Note that in both cases $|S|+1 = |I| \leq 8$.
    
    \textbf{Case 4a.} Suppose that $o(G_R)= \mathcal D$. \Cref{si-lemma:S8I9} tells us that $o(G_{SI}) \geq \mathcal N$, and it therefore follows that $o(G) \geq \mathcal N$.
    
    \textbf{Case 4b.} Suppose that there exists $u \in R$ with $\deg_{G_R}(u)=1$ and $o(G_R') = \mathcal D$, where $G_R' = G_R - u$. Add the vertex $u$ and two of its edges into $S$ to the graph $G_{SI}$ to obtain the new graph $G_{SI}'$, that is, $V(G_{SI}') = V(G_{SI}) \cup \{u\}$ and $E(G_{SI}') = E(G_{SI}) \cup \{uw_1, uw_2\}$, where $w_1, w_2 \in S \cap N_G(u)$ are distinct. Since $e(R,S)\leq 5$, there are now at most three vertices in $S$ that are of degree $2$ in $G_{SI}'$. Thus, the graph $G_{SI}'$ satisfies the conditions of \Cref{si-lemma:S7I9} so that $o(G_{SI}') \geq \mathcal N$ and thus $o(G) \geq \mathcal N$. 
    
    \textbf{Case 5.} Suppose that $7 \leq |R| \leq 12$. \Cref{si-lemma:S6I7} implies that $o(G_{SI})=\mathcal D$ and \Cref{r-lemma:R_12} implies $o(G_R) \geq \mathcal N$. Therefore, $o(G) \geq \mathcal N$.
    
    \textbf{Case 6.} Suppose that $13 \leq |R| \leq 14$. All the vertices in $I$ have degree $3$ in $G_{SI}$, so that $G_{SI}$ must be isomorphic to $K_{3,4}$, on which the Dominator wins going second. We split into further subcases depending on the arrangement of the degree $1$ vertices in $G_R$.
    
    \textbf{Case 6a.} Suppose that there are no degree $1$ vertices in $G_R$. It then follows from \Cref{r-lemma:R_14_no_deg_1} that $o(G_R) \geq \mathcal N$, and thus $o(G) \geq \mathcal N$.
    
    \textbf{Case 6b.} Suppose that there is exactly one degree $1$ vertex $u$ of $G_R$. In $G$, $u$ has at least two distinct neighbours $u_1, u_2 \in S$. Consider the graph $G_R'$ we get by adding a fresh vertex $w$ together with the edge $uw$ to $G_R$, that is, $V(G_R') = V(G_R) \cup \{w\}$ and $E(G_R') = E(G_R) \cup \{uw\}$. Note that the potential of the game state $(G_R', \{w\},\emptyset)$ is less than $1$. Indeed, since there are at most $3$ degree $2$ vertices in $G_R$, the potential is maximized if $G_R$ has the degree sequence 
    \[(1,2,2,2,3,3,3,3,3,3,3,3,3,4),\]
    in which case $\pot(G_R', \{w\},\emptyset) = \frac{3}{2^3} + \frac{9}{2^4} + \frac{1}{2^5} < 1$. By \Cref{theorem:ES_for_gs}, the Dominator has a winning strategy on $(G_R', \{w\},\emptyset)$ going first, and by \Cref{si-lemma:K34_special_strat}, the Dominator has a winning strategy on $G_{SI}$ as the second player, with which he can also guarantee to claim a vertex from $\{u_1,u_2\}$. Combining these two strategies, we see that the Dominator wins on $G$ going first. Indeed, at the end of the game all the vertices in $V(G) \setminus \{u\}$ are clearly dominated in $G$, and $u$ is dominated, since at least one of its neighbours $u_1$ or $u_2$ is claimed by the Dominator.
    
    For the remaining subcases, we may assume that there are exactly two vertices $u,u' \in R$ with degree $1$ in $G_R$.

    \textbf{Case 6c.} Suppose that $N_{G_R}(u) = \{w\} \neq \{w'\} = N_{G_R}(u')$. We move the vertices $u$ and $u'$ from $G_R$ to $G_{SI}$ to obtain new graphs $G_R'$ and $G_{SI}'$, that is, $G_R' = G_R -\{u, u'\}$ and $G_{SI}' = G[S\cup I \cup \{u, u'\}]$. We prove that $o(G_{SI}') = \mathcal D$ and that $o(G_R') \geq \mathcal N$, from which $o(G) \geq \mathcal N$ follows. Using the notation from \Cref{si-lemma:K34_with_extra_vtx}, either $G_{SI}^{(1)} - w^{(1)}$ or $G_{SI}^{(2)} - w^{(2)}$ is a spanning subgraph of $G_{SI}'$ so that by \Cref{si-lemma:K34_with_extra_vtx} we have $o(G_{SI}') = \mathcal D$. For $G_R'$, note that the potential $\pot(G_R')$ is maximized when $w$ and $w'$ have degrees $2$ and $3$ in $G_R$ and when the degree sequence of $G_R'$ is 
    \[
        (1,2,3,3,3,3,3,3,3,3,3,4).
    \]
    Since $\frac{1}{2^2} + \frac{1}{2^3} + \frac{9}{2^4} + \frac{1}{2^5} < 1$, it follows from \Cref{theorem:ES_for_gs} that $o(G_R') \geq \mathcal N$.
    
    \textbf{Case 6d. } Suppose that $N_{G_R}(u) = N_{G_R}(u') = \{w\}$ and $\deg_{G_R}(w) > 3$. As before, we move the vertices $u$ and $u'$ from $G_R$ to $G_{SI}$ so that we get the new graphs $G_R' = G_R -\{u, u'\}$ and $G_{SI}' = G[S\cup I \cup \{u, u'\}]$. Again, $o(G_{SI}') = \mathcal D$ holds by \Cref{si-lemma:K34_with_extra_vtx}. For $G_R'$, note that the potential $\pot(G_R')$ is maximized when $w$ has degree $4$ in $G_R$ and when the degree sequence of $G_R'$ is 
    \[
        (2,2,3,3,3,3,3,3,3,3,3,3).
    \]
    Since $\frac{2}{2^3} + \frac{10}{2^4} < 1$, it again follows from \Cref{theorem:ES_for_gs} that $o(G_R') \geq N$.    
    
    \textbf{Case 6e. } Suppose that $N_{G_R}(u) = N_{G_R}(u') = \{w\}$ and $\deg_{G_R}(w) \leq 3$. We move the vertices $u, u'$ and $w$ from $G_R$ to $G_{SI}$ to obtain new graphs $G_R'$ and $G_{SI}'$, that is, $G_R' = G_R -\{u, u', w\}$ and $G_{SI}' = G[S\cup I \cup \{u, u', w\}]$. Either $G_{SI}^{(1)}$ or $G_{SI}^{(2)}$ is a spanning subgraph of $G_{SI}'$ so that by \Cref{si-lemma:K34_with_extra_vtx} we have $o(G_{SI}') = \mathcal D$. The potential of $G_R'$ is maximized when $w$ has a neighbour $w' \in V(G_R')$ that has degree $2$ in $G_R$ and when the degree sequence of $G_R'$ is 
    \[
        (1,3,3,3,3,3,3,3,3,3,4).
    \] 
    Since $\frac 1{2^2} + \frac {9}{2^4} + \frac{1}{2^5} < 1$, it follows that $o(G_R') \geq \mathcal N$ and thus $o(G) \geq \mathcal N$. 
\end{proof}

\section{High Surplus: Completing Cores}\label{section:high_surplus}

In this section we finish the proof that $\beta(3) = 22$. Note that, together with \Cref{theorem:beta_3_for_sur_8} and \Cref{lemma:reduction_to_reduced_graphs}, the following theorem immediately implies this result. 

\begin{theorem}\label{theorem:beta_3_for_high_sur}
    Let $G$ be a reduced graph on at most $21$ vertices and $\mindeg(G) = 3$, and assume that $\sur(G) \geq 9$. Then $o(G) \geq \mathcal N$.
\end{theorem}

Throughout this section $G$ is a graph of minimum degree $d$, and we assume that $G$ is reduced. Recall that reduced means that no vertices of degree greater than $d$ are adjacent. Then the \emph{core} of $G$ is the bipartite spanning subgraph $C$ of $G$ with edge set
\[
    E(C) = \{uv \in E(G) : \deg_G(u) = d \text{ and } \deg_G(v) > d\}.
\]

If $P_1$ and $P_2$ are spanning subgraphs of $G$ with
\[
    E(C) \subseteq E(P_1)\subseteq E(P_2)\subseteq E(G),
\]
then $P_1$ and $P_2$ are called \emph{partially completed cores}, or \emph{\pcc s} for short, and furthermore we say that $P_2$ is a \emph{partial completion of $P_1$}. In case $P_2$ has minimum degree $d$, which happens if and only if $P_2 = G$, we say that $P_2$ is a \emph{completion} of $P_1$. In the remainder of the section, all reduced graphs are of minimum degree $d$, all cores are cores of such reduced graphs of minimum degree $d$, and all \pcc s are partial completions of such cores.

\begin{figure}[ht]
    \centering
        \begin{tikzpicture}[scale=1]
        \node[vertex] (L1) at (0,  2.5) {};
        \node[vertex] (L2) at (0,  1) {};
        \node[vertex] (L3) at (0, -1) {};
        \node[vertex] (L4) at (0, -2.5) {};
        
        \node[vertex] (R1)  at (4.0,  3.0) {};
        \node[vertex] (R2)  at (4.0,  1.0) {};
        \node[vertex] (R3)  at (4.0,  0.0) {};
        \node[vertex] (R4)  at (4.0, -2.4) {};
        
        \node[vertex] (R5)  at (5.5,  2.4) {};
        \node[vertex] (R6)  at (5.5,  1.6) {};
        \node[vertex] (R7)  at (5.5,  0.4) {};
        \node[vertex] (R8)  at (5.5, -1.0) {};
        \node[vertex] (R9)  at (5.5, -2.2) {};
        
        \node[vertex] (R10) at (7.0, -0.5) {};
        
        \draw[edge] (L1)--(R5);
        \draw[edge] (L1)--(R3);
        \draw[edge] (L1)--(R1);
        \draw[edge] (L1)--(R2);
        \draw[edge] (L1)--(R6);
        
        \draw[edge] (L2)--(R5);
        \draw[edge] (L2)--(R1);
        \draw[edge] (L2)--(R2);
        \draw[edge] (L2)--(R4);
        \draw[edge] (L2)--(R7);
        
        \draw[edge] (L3)--(R3);
        \draw[edge] (L3)--(R1);
        \draw[edge] (L3)--(R4);
        \draw[edge] (L3)--(R8);
        \draw[edge] (L3)--(R9);
        
        \draw[edge] (L4)--(R3);
        \draw[edge] (L4)--(R2);
        \draw[edge] (L4)--(R4);
        \draw[edge] (L4)--(R7);
        \draw[edge] (L4)--(R8);
        
        \draw[edge,dashed] (R10)--(R9);
        \draw[edge,dashed] (R10)--(R7);
        \draw[edge,dashed] (R10)--(R6);
        
        \draw[edge,dashed] (R8)--(R9);
        \draw[edge,dashed] (R5)--(R6);
    \end{tikzpicture}
    \caption{A reduced graph $G$ of minimum degree $3$ and its core $C(G)$. The dashed edges belong to $G$ but not to $C(G)$; all the other edges belong to both $G$ and $C(G)$.}
    \label{fig:placeholder}
\end{figure}
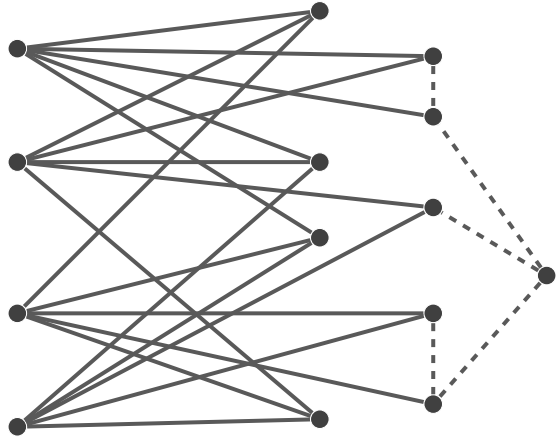

The rough structure of the proof of \Cref{theorem:beta_3_for_high_sur} is as follows. First, we use \Cref{algorithm:sequence_generator} to generate all degree sequences of reduced graphs on at most $21$ vertices of minimum degree $3$ and surplus at least $9$, which are not ruled out by \Cref{lemma:two_step_potential}, an easy application of the Erd\H os-Selfridge theorem. Using \nauty, we then generate the cores of the reduced graphs having one of these degree sequences. Since cores are bipartite and we have quite restrictive degree conditions, generating the cores is much easier than generating the reduced graphs themselves. Finally, we use \Cref{algorithm:completion_outcome} to check whether the Dominator wins on all completions of these cores going first.

The following lemma should be seen as a two-step analogue to \Cref{lemma:one_step_deg_seq} for degree sequences of reduced graphs. Instead of bounding the potential after one step into the game, we bound it after two steps. 

\begin{lemma}[Two-step Erd\H os-Selfridge for degree sequences of reduced graphs]\label{lemma:two_step_potential}
    Let $G$ be a reduced graph of minimum degree $d$, and assume that $N:=|V(G)| \leq 2^{d+3} - 8d - 8$. Let $(d_1,\dots,d_N)$ be the degree sequence of $G$, and let $n$ be minimal with $d_n > d$. If
    \begin{align}\label{align:two_step_potential}
        \max_{n \leq i \leq N-1}\ \ \frac{n-1-d_N+\min(d_i, n-1 - d_N)}{2^{d+1}}+ 2^{-d_i} + \sum_{\substack{j = n\\ j \neq i}}^{N-1} 2^{-d_j - 1} < 1
    \end{align}
    then $o(G) \geq \mathcal N$.
\end{lemma}
\begin{proof}
    Let $v \in V(G)$ be a vertex of degree $d_N$. We show that for every $w \in V(G) \setminus \{v\}$ we have $\pot(G, \{v\}, \{w\}) < 1$, from which $o(G) \geq \mathcal N$ follows from the Erd\H{o}s-Selfridge Theorem. Let $A,B \subseteq V(G)$ denote the sets of vertices of degree greater than $d$ and of degree $d$, respectively. First, if $w \in A \setminus \{v\}$ is of degree $d_i$ with $n \leq i \leq N-1$, then
    \begin{align*}
        \pot(G, \{v\}, \{w\}) &= \frac{|B \setminus (N(v) \cup N(w))|}{2^{d+1}}+\frac{|N(w) \setminus N(v)|}{2^{d}}+ 2^{-d_i} + \sum_{\substack{j = n \\ j \neq i}}^{N-1} 2^{-d_j - 1}.
    \end{align*}
    Since $N(v), N(w) \subseteq B$, we have that 
    \begin{align*}
        |B \setminus (N(v) \cup N(w))| + 2|N(w) \setminus N(v)| &= |B \setminus N(v)| + |N(w) \setminus N(v)| \\
        &\leq n-1-d_N + \min(d_i, n-1 - d_N)
    \end{align*}
    so that by \Cref{align:two_step_potential} we indeed get $\pot(G,\{v\}, \{w\}) < 1$. Let us now assume that $w \in B$. Note that we have 
    \begin{align*}
        &|B \setminus (N(v) \cup N[w])| + |N[w] \cap B \setminus N(v)| \leq n-1-d_N \quad \text{and}  \\
        &|N[w] \cap B \setminus N(v)|+ |A \cap N(w) \setminus \{v\}| \leq d+1,
    \end{align*}
    so that
    \begin{align*}
        \pot(G, \{v\}, \{w\}) &= \frac{|B \setminus (N(v) \cup N[w])|}{2^{d+1}} + \frac{|N[w] \cap B \setminus N(v)|}{2^d}\\ 
        &+ \sum_{u \in A \cap N(w) \setminus \{v\}} 2^{-\deg(u)} + \sum_{u \in A \setminus (N(w) \cup \{v\})} 2^{-\deg(u) - 1}\\
        &\leq \frac{|B \setminus (N(v) \cup N[w])| + 2|N[w] \cap B \setminus N(v)| + |A \cap N(w) \setminus \{v\}|}{2^{d+1}}+ \sum_{j = n}^{N-1} 2^{-d_j - 1}\\
        &\leq \frac{n-d_N+d}{2^{d+1}} + \sum_{\substack{j = n}}^{N-1} 2^{-d_j - 1}
    \end{align*}
    follows.
    Now, if $n \geq d_N + d + 2$, then 
    \begin{align*}
        \frac{n-d_N+d}{2^{d+1}} + \sum_{j = n}^{N-1} 2^{-d_j - 1} \leq \frac{n-1-d_N+\min(d_{N-1}, n - 1 - d_N)}{2^{d+1}}+ 2^{-d_{N-1}} + \sum_{j=n}^{N-2} 2^{-d_j - 1} < 1
    \end{align*}
    and if $n < d_N + d + 2$, that is, $n - d_N + d < 2d+2$, then
    \begin{align*}
        \frac{n-d_N+d}{2^{d+1}} + \sum_{j = n}^{N-1} 2^{-d_j - 1} < \frac{d+1}{2^{d}} + \frac{N/2}{2^{d+2}} \leq 1.
    \end{align*}
    Here the first inequality follows from $n \geq N/2$, which holds since $G$ is reduced, and the second inequality follows from the assumption $N\leq 2^{d+3} - 8d - 8$. In every case we indeed have $\pot(G,\{v\}, \{w\}) < 1$, which finishes the proof.
\end{proof}

\begin{algorithm}[ht]
    \caption{\textsc{GenerateSeqs}$(N, d, s)$}
    \label{algorithm:sequence_generator}
    \begin{algorithmic}[1]
        \REQUIRE Positive integers $N, d, s$ with $d < N$
        \ENSURE A set containing all degree sequences of reduced graphs on $N$ vertices of minimum degree $d$ and surplus at least $s$, for which \Cref{align:two_step_potential} does not hold
        
        \STATE $L \gets \emptyset$
        \STATE $(d_1,\dots,d_N)\gets (d,\dots,d,d+1,d+1)$ \label{algorithm_line:sequence_generator_seq_instanciation}
        \WHILE{$d_1 \neq d$}
            \STATE $n \gets \min\{i \in [N] : d_i >d\}$
            \IF{$d_N < N$ and \Cref{align:two_step_potential} does not hold}\label{algorithm_line:sequence_generator_if_inequality_does_not_hold}
                \IF{$\sum_{j=n}^{N-1} (d_j -d) \geq s$ and $2\mid\big((n-1)\cdot d + \sum_{j=n}^{N} d_j\big)$ and $(n-1)\cdot d \geq \sum_{j=n}^{N}d_j$}\label{algorithm_line:sequence_generator_if_deg_seq_conditions}
                    \STATE $L \gets L \cup \{(d_1,\dots,d_N)\}$ \label{algorithm_line:sequence_generator_add_seq}
                \ENDIF
                \STATE $d_N \gets d_N + 1$ \label{algorithm_line:sequence_generator_inc_last}
            \ELSE
                \STATE $k \gets \max\{ j \in [N] : d_j \neq d_N\}$
                \STATE $(d_k,d_{k+1}, \dots, d_N) \gets (d_k +1, \dots, d_k + 1)$ \label{algorithm_line:sequence_generator_inc_last_few}
            \ENDIF
        \ENDWHILE
        \STATE \textbf{return} $L$
    \end{algorithmic}
\end{algorithm}

\Cref{algorithm:sequence_generator} lexicographically iterates over all potential degree sequences of minimum degree $d$ and checks whether 
\begin{itemize}
    \item \Cref{align:two_step_potential} does not hold,
    \item the surplus is at least some fixed $s$,
    \item the sum of the degrees is even,
    \item the degree sequence can be of a reduced graph.
\end{itemize}
and returns a list of the remaining degree sequences.

The implementation of \Cref{algorithm:sequence_generator} which we use to generate the degree sequences of the reduced graphs whose cores we later want to check can be found in \texttt{sequence\_generator.cpp}. 

\begin{proof}[Proof of correctness of \Cref{algorithm:sequence_generator}]
    First observe that whenever the sequence $(d_1, \dots, d_N)$ is defined, it is non-decreasing with $d_N \leq N$. Indeed, $(d_1,\dots,d_N)$ is only ever updated in \lineref{algorithm_line:sequence_generator_inc_last} and \lineref{algorithm_line:sequence_generator_inc_last_few}, both updates preserve monotonicity, and $d_N$ only increases in \Cref{algorithm_line:sequence_generator_inc_last}, in which case $d_N < N$ by the if query in \lineref{algorithm_line:sequence_generator_if_inequality_does_not_hold}. Furthermore, observe that the while loop runs through the sequences $(d_1,\dots,d_N)$ in lexicographic order; that is, if in some iteration of the loop the sequence is $(s_1,\dots,s_N)$ and in the next iteration it is $(s_1',\dots,s_N')$ then $(s_1, \dots, s_N) <_{lex} (s_1',\dots,s_N')$. Here $(s_1, \dots, s_N) <_{lex} (s_1',\dots,s_N')$ if there is $k \in [N]$ with $s_i = s_i'$ for $i < k$ and $s_{k} < s_{k}'$. In particular, the while loop terminates as there are only finitely many non-decreasing sequences $(d_1,\dots,d_N)$ with $d_N \leq N$.
    
    Now let $\mathbf s=(s_1,\dots, s_n)$ be a degree sequence of a reduced graph $G$ of minimum degree $d$ and $\sur(G) \geq s$, for which \Cref{align:two_step_potential} does not hold. We have to show that at some point $\mathbf s$ is added to $L$ in \lineref{algorithm_line:sequence_generator_add_seq}. By assumption $s_N < N$, \Cref{align:two_step_potential} does not hold for $\mathbf s$, and the three conditions in \lineref{algorithm_line:sequence_generator_if_deg_seq_conditions} are satisfied for $\mathbf s$: the first one follows from $\sur(G) \geq s$, the second one from the handshaking lemma, and the third one from the fact that $G$ is reduced. This shows that in fact it is enough to prove that $(d_1,\dots,d_N) = \mathbf s$ in some iteration of the while loop. Let $\mathbf a=(a_1,\dots,a_N)<_{lex} \mathbf s$ be maximal with respect to $<_{lex}$ such that $(d_1,\dots,d_N) = \mathbf a$ in some iteration of the while loop, and let $\mathbf b$ denote the sequence $(d_1,\dots,d_N)$ in the next iteration.
    
    By way of contradiction, we assume that $\mathbf s <_{lex} \mathbf b$. If $a_N < N$ and \Cref{align:two_step_potential} does not hold for $\mathbf a$, then $\mathbf b = (a_1,\dots, a_{N-1}, a_N+1) \leq_{lex} \mathbf s$, contradicting $\mathbf s <_{lex} \mathbf b$. So from now on assume that $a_N \geq N$ or that \Cref{align:two_step_potential} holds for $\mathbf a$. Let $k \leq N-1$ be maximal with $a_k < a_N$. Since $\mathbf a <_{lex} \mathbf s <_{lex} \mathbf b$, the three sequences $\mathbf a, \mathbf s$, and $\mathbf b$ must be of the form
    \begin{align*}
    \begin{matrix}
        \mathbf a &= (a_1, &\dots &a_{k-1}, &a_k, &a_N, &a_N, &\dots &a_N)\\ 
        \mathbf s &= (a_1, &\dots &a_{k-1}, &a_k, &s_{k+1}, &s_{k+2}, &\dots &s_N)\\ 
        \mathbf b &= (a_1, &\dots &a_{k-1}, &a_k+1, &a_k+1, &a_k+1, &\dots &a_k+1)
    \end{matrix}
    \end{align*}
    with $s_{k+1} \geq a_N$. This proves that $a_i \leq s_i$ for all $i \in [N]$ and $a_N < s_N$. Note that since \Cref{align:two_step_potential} does not hold for $\mathbf s$, it also does not hold for the component-wise smaller sequence $\mathbf a$. Finally, we also can not have $a_N\geq N$ since then $s_N \geq N$, which implies the desired contradiction.
\end{proof}

The second algorithm we need for the proof of \Cref{theorem:beta_3_for_high_sur} is \Cref{algorithm:completion_outcome}. Given a core, or more generally a \pcc\ $P$, and a first player $p$, it decides whether the Dominator wins on all completions of $P$ when $p$ starts the game. Roughly speaking, the algorithm proceeds by inserting edges to $P$ one-by-one and after each insertion checking if the Dominator already wins on all the completions using \Cref{algorithm:potential_outcome}. For the latter algorithm, we need the notion of a game state on a \pcc. Let $P$ be a \pcc\ and let 
\begin{align}\label{align:claimed_vertices_in_pccgs_high_deg}
    D, S \subseteq \{v \in V(P) : \deg_P(v) \geq d\}
\end{align}
be disjoint. Then the triple $(P,D,S)$ is called \emph{partially completed core game state on $P$} or \emph{\pccgs} for short. Note, that in the supporting code, we use \texttt{PCCGS}. The game state degree of a vertex of a \pccgs\ and the potential of a \pccgs\ are defined via the game state degree and the potential of its completions.
Let $G$ be a completion of the \pcc\ $P$ and let $v$ be a vertex of $P$. Then the \emph{game state degree of $v$ with respect to $(P,D,S)$} and the \emph{potential} of the \pccgs\ $(P,D,S)$ are defined by
\[
    \pccgsdeg_{(P,D,S)}(v) = \gsdeg_{(G,D,S)}(v) \text{ and } \pccpot(P,D,S) = \pot(G,D,S).
\]
Because of \Cref{align:claimed_vertices_in_pccgs_high_deg}, $D$ and $S$ only contain vertices $v$ with $N_P(v) = N_G(v)$, so that both quantities are well-defined as they do not depend on the completion $G$. Similarly, as for \pcc s if $(P_1,D,S)$ and $(P_2, D,S)$ are \pccgs s such that $P_2$ is a (partial) completion of $P_1$, then $(P_2,D,S)$ is a (partial) completion of $(P_1, D,S)$.

Before stating \Cref{algorithm:potential_outcome} and proving its correctness, note that the main structure closely follows that of \Cref{algorithm:outcome_game_state}. Let us informally discuss some differences between these algorithms. The main difference is that we are given a \pccgs\ $(P,D,S)$ instead of a game state. In other words, some edges between vertices $v\in V(P)$ with $\deg_P(v) < d$ are “hidden”. Accordingly, we cannot always guarantee to decide whether the Dominator wins on all completions of $(P,D,S)$ without making assumptions about these hidden edges, in which case the algorithm returns \unsure. It also has as a consequence that the Dominator is only allowed to pick vertices $v$ of degree $\deg_P(v) \geq d$, since these are vertices whose neighbourhood $N_P(v)$ is the same in all completions of $P$. Finally, in the case of the Staller claiming a vertex $v \in V(P)$ with $\deg_P(v) < d$, the Erd\H os-Selfridge potential in completions of $(P,D,S)$ may depend on how we complete the hidden edges containing $v$. In this case, we must give an upper bound for the potential by assuming the potential-wise worst hidden neighbours of $v$. 

\begin{algorithm}[H]
    \caption{\textsc{PotentialOutcome}$((P,D,S),p)$}
    \label{algorithm:potential_outcome}
    \begin{algorithmic}[1]
        \REQUIRE A \pccgs\ $(P,D,S)$ and a first player $p \in \{\mathcal D, \mathcal S\}$
        \ENSURE Returns $\mathcal D$, $\mathcal S$ or \unsure. If it returns $\mathcal D$ then $o_p(G,D,S) = \mathcal D$ for every completion $G$ of $P$. If it returns $\mathcal S$ then $o_p(G,D,S) = \mathcal S$ for some completion $G$ of $P$. If $P$ is its own completion, that is $\mindeg(P) = d$, then it does not return \unsure. 
        \STATE $L \gets \{u \in V(P) : \deg_{P}(u) < d\}$ \COMMENT{uncompleted vertices}
        \FORALL{$v \in V(P)$}
            \STATE \textbf{if} $v \in L$ and $|L \setminus N_P[v]| + \deg_P(v) < d$ \textbf{then} \textbf{return} $\mathcal D$ \label{algorithm_line:potential_outcome_no_completion}
            \STATE \textbf{if} $\pccgsdeg_{(P,D,S)}(v) = 0$ \textbf{then} \textbf{return} $\mathcal S$ \label{algorithm_line:potential_outcome_Staller_already_won}
        \ENDFOR
        \IF{$p = \mathcal D$}
            \STATE \textbf{if} $\pccpot(P,D,S) < 1$ \textbf{then} \textbf{return} $\mathcal D$ \label{algorithm_line:potential_outcome_Dominator_potreturn}
            \STATE \textbf{if} $L \cup D \cup S = V(G)$ and $L \neq \emptyset$ \textbf{then} \textbf{return} \unsure \label{algorithm_line:potential_outcome_Dominator_has_no_move_unsure}
            \STATE $uFlag \gets$ \textbf{false}
            \FORALL{$v \in V(P) \setminus (L \cup D \cup S)$}
                \STATE $out \gets \textsc{PotentialOutcome}((P, D \cup \{v\}, S), \mathcal S)$
                \STATE \textbf{if} $out = \mathcal D$ \textbf{then} \textbf{return} $\mathcal D$
                \STATE \textbf{if} $out = \unsure$ \textbf{then} $uFlag \gets$ \textbf{true}
            \ENDFOR
            \STATE \textbf{if} $uFlag$ \textbf{then} \textbf{return} \unsure\ \textbf{else} \textbf{return} $\mathcal S$ \label{algorithm_line:potential_outcome_Dominator_final_return}
        \ELSIF{$p = \mathcal S$}
            \STATE \textbf{if} $\pccpot(P,D,S) < \frac{1}{2}$ \textbf{then} \textbf{return} $\mathcal D$ \label{algorithm_line:potential_outcome_Staller_potreturn}
            \FORALL{$v \in L$} \label{algorithm_line:potential_outcome_Staller_L_loop}
                \STATE Let $l_1,\dots,l_m$ be pairwise distinct such that $\{l_1,\dots,l_m\} = L \setminus N_P[v]$ and \label{algorithm_line:potential_outcome_Staller_L_loop_l_def}
                \Statex \quad\quad $\pccgsdeg_{(P,D,S)}(l_1) \leq \dots \leq \pccgsdeg_{(P,D,S)}(l_m)$
                \STATE $N \gets N_P[v] \cup \{l_1, \dots, l_{d - \deg_P(v)}\}$
                \STATE \textbf{if} $\pccpot(P,D,S) + \sum_{u \in N} 2^{-\pccgsdeg_{(P,D,S)}(u)} \geq 1$ \textbf{then} \textbf{return} \unsure\label{algorithm_line:potential_outcome_Staller_L_loop_return}
            \ENDFOR
            \STATE $uFlag \gets$ \textbf{false}
            \FORALL{$v \in V(P) \setminus (L \cup D \cup S)$}
                \STATE $out \gets \textsc{PotentialOutcome}((P, D, S \cup \{v\}), \mathcal D)$ \label{algorithm_line:potential_outcome_Staller_rec_call}
                \STATE \textbf{if} $out = \mathcal S$ \textbf{then} \textbf{return} $\mathcal S$ \label{algorithm_line:potential_outcome_Staller_return_after_rec_call}
                \STATE \textbf{if} $out = \unsure$ \textbf{then} $uFlag \gets$ \textbf{true}
            \ENDFOR
            \STATE \textbf{if} $uFlag$ \textbf{then} \textbf{return} \unsure\ \textbf{else} \textbf{return} $\mathcal D$ \label{algorithm_line:potential_outcome_Staller_final_return}
        \ENDIF
    \end{algorithmic}
\end{algorithm}

\begin{proof}[Proof of correctness of \Cref{algorithm:potential_outcome}]
    Let $(P,D,S)$ be a \pccgs, let $p \in \{\mathcal S, \mathcal D\}$, and let $L = \{u \in V(P) : \deg_{P}(u) < d\}$. If there is $v\in L$ with $|L \setminus N_P[v]| + \deg_P(v)< d$, then there is no completion of $P$, and hence we may return $\mathcal D$ in \lineref{algorithm_line:potential_outcome_no_completion}. Now assume there is a completion $G$ of $P$. Fix an arbitrary such completion $G$ for the rest of the proof. If there is $v \in V(P)$ with $\gsdeg_{(G,D,S)}(v) = \pccgsdeg_{(P,D,S)}(v)= 0$, then the Staller wins on $(G,D,S)$ regardless of the first player, and we return $\mathcal S$ in \lineref{algorithm_line:potential_outcome_Staller_already_won}. From now on, assume $\pccgsdeg_{(P,D,S)}(v) > 0$ for all $v \in V(P)$ and $|L \setminus N_P[v]| + \deg_P(v)\geq d$ for all $v \in L$.
    
    As in the proof of correctness of \Cref{algorithm:outcome_game_state}, we use induction. But before, it is convenient to split off a lemma that we use both in the induction base and the induction step. 

    \begin{lemma}\label{lemma:potential_outcome_low_deg_vtx_lemma}
        If the algorithm reaches the for loop in \Cref{algorithm_line:potential_outcome_Staller_L_loop} but does not return \unsure\ in \Cref{algorithm_line:potential_outcome_Staller_L_loop_return}, then the Dominator wins on $(G,D,S\cup \{v\})$ going first for all $v \in L$.
    \end{lemma}
    \begin{proof}
        Let $v \in L$ and let $l_1,\dots,l_m \in L$ be such as in \Cref{algorithm_line:potential_outcome_Staller_L_loop_l_def}, that is, they are pairwise different with
    \[
        \{l_1,\dots,l_m\} = L \setminus N_P[v] \ \text{ and } \ \pccgsdeg_{(P,D,S)}(l_1) \leq \dots \leq \pccgsdeg_{(P,D,S)}(l_m).
    \]
    Note that $m = |L \setminus N_P[v]| \geq d - \deg_P(v)$ and let $N =  N_P[v] \cup \{l_1, \dots, l_{d - \deg_P(v)}\}$. Then 
    \begin{align*}
        \pot(G,D,S\cup \{v\}) =& \pot(G,D,S) + \sum_{u \in N_{G}[v]} 2^{-\gsdeg_{(G,D,S)}(u)}\\ 
        =& \pccpot(P,D,S) + \sum_{u \in N_P[v]} 2^{-\pccgsdeg_{(P,D,S)}(u)}\\ 
        &+ \sum_{u \in N_G(v) \setminus N_P(v)}2^{-\pccgsdeg_{(P,D,S)}(u)}
    \end{align*}
    The latter sum can be bounded by $\sum_{i = 1}^{d- \deg_P(v)}2^{-\pccgsdeg_{(P,D,S)}(l_i)}$ since $N_G(v) \setminus N_P(v) \subseteq \{l_1,\dots,l_m\}$ so that all in all we get 
    \[\pot(G,D,S\cup \{v\}) \leq \pccpot(P,D,S) + \sum_{u \in N} 2^{-\pccgsdeg_{(P,D,S)}(u)}.\]
    In case we never return \unsure\ in \Cref{algorithm_line:potential_outcome_Staller_L_loop_return}, the latter is less than $1$ so that indeed the Dominator wins on $(G,D,S\cup \{v\})$ going first by the Erd\H os-Selfridge Theorem.
    \end{proof}

    Let us now proceed by induction on $n := |V(P) \setminus (L \cup D \cup S)|$. First consider the base case $n = 0$. If $L = \emptyset$, that is, $P$ is its own completion, then $\pccpot(P,D,S) \in \Z_{\geq 0}$ and thus $\pccpot(P,D,S) = 0$ since $\pccgsdeg_{(P,D,S)}(v)> 0$ for all $v \in V(G)$. Regardless of the first player $p$, the algorithm returns $\mathcal D$ in \Cref{algorithm_line:potential_outcome_Dominator_potreturn} or \Cref{algorithm_line:potential_outcome_Staller_potreturn}, and indeed the Dominator wins on $(G,D,S) = (P,D,S)$. Now if $L \neq \emptyset$ and $p = \mathcal D$, then we return \unsure\ in \Cref{algorithm_line:potential_outcome_Dominator_has_no_move_unsure}, which is allowed since $P$ is not its own completion. Finally, if $L \neq \emptyset$ and $p = \mathcal S$, then by \Cref{lemma:potential_outcome_low_deg_vtx_lemma} the algorithm returns $\mathcal D$ in \Cref{algorithm_line:potential_outcome_Staller_final_return} only if he wins on $(G,D,S)$. If we do not return $\mathcal D$ in \Cref{algorithm_line:potential_outcome_Staller_final_return}, then returning \unsure\ in \Cref{algorithm_line:potential_outcome_Staller_L_loop_return} is allowed since $L \neq \emptyset$.

    For the induction step, assume that $n > 0$ and that \Cref{algorithm:potential_outcome} matches its specification for \pccgs s $(P',D',S')$ with $|V(P') \setminus (L' \cup D' \cup S')| < n$, where $L' = \{u \in V(P') : \deg_{P'}(u) < d\}$. Let us only do the case $p = \mathcal S$. The other case is almost analogous. If $\pot(G,D,S) = \pccpot(P,D,S) < \frac 12$, then the algorithm returns $\mathcal D$ in \Cref{algorithm_line:potential_outcome_Staller_potreturn}, and indeed the Dominator wins on $(G,D,S)$ going second. If we return \unsure\ in \Cref{algorithm_line:potential_outcome_Staller_L_loop_return}, then $L \neq \emptyset$ so that indeed $P$ can not be its own completion. By \Cref{lemma:potential_outcome_low_deg_vtx_lemma}, we can now assume that the Dominator wins on $(G,D,S\cup \{v\})$ going first for all $v \in L$.
    
    For a vertex $v \in V(P) \setminus (L \cup D \cup S)=:R$ let $o(v) = \textsc{PotentialOutcome}((P, D, S \cup \{v\}), \mathcal D)$, and note that the recursive call in \Cref{algorithm_line:potential_outcome_Staller_rec_call} correctly computes $o(v)$ by the induction hypothesis. If $o(v) = \mathcal S$ for some $v \in R$, then the algorithm returns $\mathcal S$ in \Cref{algorithm_line:potential_outcome_Staller_return_after_rec_call}, and indeed she wins on $(G,D,S)$ going first. If $o(v) = \mathcal D$ for all $v \in R$, the algorithm returns $\mathcal D$, and indeed for all $w \in V(P) \setminus (D \cup S)$ the Dominator wins on $(G,D,S\cup \{w\})$ going first. Otherwise, the algorithm returns \unsure\ and we must have $o(v) = \unsure$ for some $v \in R$, in which case we indeed have $L \neq \emptyset$ by the induction hypothesis. 
\end{proof}

\begin{algorithm}[H]
    \caption{\textsc{CompletionOutcome}$(P, p)$}
    \label{algorithm:completion_outcome}
    \begin{algorithmic}[1]
        \REQUIRE A \pcc\ $P$ and a first player $p \in \{\mathcal D, \mathcal S\}$
        \ENSURE Returns $\mathcal D$ if $o_p(G) = \mathcal D$ for every completion $G$ of $P$; returns $\mathcal S$ otherwise.
        \STATE $S \gets \{P\}$ \label{algorithm_line:compl_out_while_loop}
        \WHILE{$S \neq \emptyset$}
            \STATE $S' \gets S$
            \STATE $S \gets \emptyset$
            \FORALL{$P' \in S'$} \label{algorithm_line:compl_out_for_loop}
                \STATE $out \gets \textsc{PotentialOutcome}((P', \emptyset, \emptyset), p)$ \label{algorithm_line:compl_out_potout_call}
                \STATE \textbf{if} $out = \mathcal S$ \textbf{then} \textbf{return} $\mathcal S$ \label{algorithm_line:compl_out_return_S}
                \IF{$out = \unsure$}
                    \STATE $L \gets \{u \in V(P') : \deg_{P'}(u) < d\}$
                    \STATE Let $v \in L$ \label{algorithm_line:compl_out_pick_low_deg_vtx}
                    \FORALL{$w \in L \setminus N_{P'}[v]$} \label{algorithm_line:compl_out_for_loop_over_potlow_deg_neigh}
                        \STATE $S \gets S \cup \big\{(V(P'), E(P') \cup \{vw\})\big\}$ \label{algorithm_line:compl_out_add_new_compl_to_S}
                    \ENDFOR
                \ENDIF
            \ENDFOR
        \ENDWHILE
        \STATE \textbf{return} $\mathcal D$ \label{algorithm_line:compl_out_return_D}
    \end{algorithmic}
\end{algorithm}
\begin{proof}[Proof of correctness of \Cref{algorithm:completion_outcome}]
    In this proof we always assume that $p$ starts the MBD-game. First, note that every completion of $P$ has 
    \[
        |E(P)| + \frac12\sum_{v \in V(P)} \max\{0, d - \deg_{P}(v)\}
    \]
    edges. For a completion $G$ of $P$, let $m = |E(G)| - |E(P)|$, and for $n \geq 0$ denote by $S_n$ the set $S$ of \pcc s after the $n$-th iteration of the while loop in \lineref{algorithm_line:compl_out_while_loop}, so that in particular $S_0 = \{P\}$. In case the algorithm returns before the $n$-th iteration of the while loop ends or if it never reaches the $n$-th iteration of the while loop, $S_n$ is undefined. Note that if $S_n$ is defined, it consists of partial completions of $P$ that have $|E(P)| + n$ edges. This means that if we return $\mathcal S$ in \lineref{algorithm_line:compl_out_return_S}, it must be because $\textsc{PotentialOutcome}((P',\emptyset,\emptyset), p) = \mathcal S$ for some partial completion $P'$ of $P$. In this case it follows from the specification of \lineref{algorithm:potential_outcome} that we indeed have a completion of $P'$, and thus of $P$, on which the Staller wins.

    We now show by induction on $n$ that if $S_n$ is defined and if $G$ is a completion of $P$ on which the Dominator loses, then there is $P_2 \in S_n$ such that $G$ is a completion of $P_2$. The base case $n=0$ is trivial, so let us consider the case $n \geq 1$.

    Let $G$ be a completion of $P$ on which the Dominator loses, and assume that $S_n$ is defined. By the induction hypothesis, there is a $P_1\in S_{n-1}$ such that $G$ is a completion of $P_1$. Consider the $n$-th iteration of the while loop. In \lineref{algorithm_line:compl_out_for_loop} we run over all the elements $P'$ of $S_{n-1}$. Consider the iteration of the for loop in which we run over $P' = P_1$. Since $P_1$ has a completion, namely $G$, on which the Dominator loses, we cannot have $out = \mathcal D$ in \lineref{algorithm_line:compl_out_potout_call} so that $out \in \{\mathcal S, \unsure\}$. Since $S_n$ is defined, we do not return $\mathcal S$ in \lineref{algorithm_line:compl_out_return_S} so that $out = \unsure$. The vertex $v$ we pick in \lineref{algorithm_line:compl_out_pick_low_deg_vtx} has degree less than $d$ in $P_1$. In $G$ it thus has a neighbour $w \in N_G(v) \setminus N_{P_1}(v)$ that we iterate over in the for loop in \lineref{algorithm_line:compl_out_for_loop_over_potlow_deg_neigh}. Therefore, $G$ is a completion of the \pcc\ $P_2=(V(P_1), E(P_1) \cup \{vw\})$. This $P_2$ is added to $S$ in \lineref{algorithm_line:compl_out_add_new_compl_to_S} and thus is an element of $S_n$, which finishes the induction step. 

    We already dealt with the case of the algorithm returning $\mathcal S$ in \lineref{algorithm_line:compl_out_return_S}, so let us assume it does not. The number of iterations of the while loop is bounded by $m+1$, since $S_{m+1} = \emptyset$ in case it is defined. In particular, $S_n = \emptyset$ for some $n \geq 1$, so that by the above observation there can not be a completion of $P$ on which the Dominator loses. The algorithm also returns $\mathcal D$ in \lineref{algorithm_line:compl_out_return_D}.
\end{proof}

\begin{proof}[Proof of \Cref{theorem:beta_3_for_high_sur}] 
    By \Cref{lemma:reduction_to_reduced_graphs} and \Cref{theorem:ES_for_gs} we only have to deal with reduced graphs of minimum degree $d=3$, surplus at least $s=9$, and $16 \leq N \leq 21$ vertices.
    Running \Cref{algorithm:sequence_generator} with this $d$ and $s$ and all $16 \leq N \leq 21$, we obtain a list of $118$ degree sequences, which can be found in \texttt{input\_sequences.txt}.
    We bunch up these degree sequences into $38$ classes such that for every class there is a single \nauty\ command that generates all cores of graphs corresponding to this given degree sequence. These \nauty\ commands can be found in \texttt{cores\_cmds.txt}.
    Using \Cref{algorithm:completion_outcome}, we check all cores generated by all of these \nauty\ commands in \texttt{generate\_and\_check.sh}.
    In order to parallelize, all core generations are split into blocks as provided by \geng. In total \pgfmathprintnumber{3756845840} cores need to be checked which needs around 3184 hours of computation.
\end{proof}

\addcontentsline{toc}{section}{Acknowledgments}
\section*{Acknowledgements} Jakob Führer, Paul Hametner and Oliver Roche-Newton were supported by the Austrian Science Fund FWF Project 10.55776/PAT2559123.
This article is partially based upon work from COST Action CA22145, supported by COST (European Cooperation in Science and Technology).

This research was funded in part by the Austrian Science Fund (FWF) [10.55776/PAT2559123]. For the purpose of open access, the authors have applied a CC BY public copyright licence to any Author Accepted Manuscript version arising from this submission.

We sincerely thank Valentin Gledel for sharing his insights with us concerning the problems in this paper, and especially for showing us his construction giving the upper bound for $\beta(d)$ in \Cref{cor:g-bound}, which appears for the first time in this paper, and for which we claim no credit. We are very grateful to Qi Jiayue for introducing us to the Maker Breaker Domination Game and for many helpful conversations on the topic, and to Patrick Mederitsch for some earlier help with attempts to solve this problem via coding. We also thank Krishnendu Bhowmick, Laura Dilly and Miriam Patry for many helpful conversations.

A variant of this work was published as Paul Hametner's master's thesis at Johannes Kepler University Linz.

\bibliographystyle{plainurl}
\bibliography{references}

@article{
    BDGLP,
    author = "Guillaume Bagan and Eric Duchêne and Valentin Gledel and Tuomo Lehtilä and Aline Parreau",
    title = "Partition strategies for the Maker-Breaker domination game",
    journal = "Algorithmica",
    volume = "87",
    pages = "191--222",
    year = "2025",
    doi = "10.1007/s00453-024-01280-x"
}

@article {
    BKR,
    AUTHOR = {Bre\v{s}ar, Bo\v{s}tjan and Klav\v{z}ar, Sandi and Rall, Douglas F.},
    TITLE = {Domination game and an imagination strategy},
    JOURNAL = {SIAM Journal on Discrete Mathematics},
    VOLUME = {24},
    YEAR = {2010},
    NUMBER = {3},
    PAGES = {979--991},
    DOI = {10.1137/100786800},
}

@article{
    DGPR,
    author = "Eric Duchêne and Valentin Gledel and Aline Parreau and Gabriel Renault",
    title = "Maker-breaker domination game",
    journal = "Discrete Math.",
    volume = "343",
    pages = "12 pp.",
    year = "2020",
    doi = "10.1016/j.disc.2020.111955"
}

@article{
    ErdösSelfridge,
    author = {Erdős, Paul and Selfridge, John L.},
    title = {On a combinatorial game},
    journal = {Journal of Combinatorial Theory, Series A},
    volume = {14},
    number = {3},
    pages = {298--301},
    year = {1973},
    doi = {10.1016/0097-3165(73)90005-8}
}

@book{
    HKSS,
    author = {Hefetz, Dan and Krivelevich, Michael and Stojakovi\'{c}, Milo\v{s} and Szab\'{o}, Tibor},
    title = {Positional games},
    series = {Oberwolfach Seminars},
    volume = {44},
    publisher = {Springer},
    address = {Basel},
    year = {2014},
    doi = {10.1007/978-3-0348-0825-5}
}

@article{
    McKayPiperno2014,
    author = {Brendan D. McKay and Adolfo Piperno},
    title = {Practical graph isomorphism, {II}},
    journal = {Journal of Symbolic Computation},
    volume = {60},
    pages = {94--112},
    year = {2014},
    doi = {10.1016/j.jsc.2013.09.003},
}

@article{
    Tutte_1953,
    author = {Tutte, William T.},
    title = {The 1-factors of oriented graphs},
    journal = {Proceedings of the American Mathematical Society},
    volume = {4},
    number = {6},
    pages = {922--931},
    year = {1953},
    doi = {10.1090/s0002-9939-1953-0063009-7}    
}

@article{
    Tutte1947,
    author = {William T. Tutte},    
    title = {The Factorization of Linear Graphs},
    journal = {Journal of The London Mathematical Society-second Series},
    year = {1947},
    pages = {107--111},
    doi = {10.1112/jlms/s1-22.2.107}
}

@book{
    Akiyama_Kano_book,
    author = "Jin Akiyama and Mikio Kano",
    title = "Factors and factorizations of graphs: Proof techniques in factor theory",
    pages = "1--367",
    year = "2011",
    series = "Lecture Notes in Mathematics",
    publisher = "Springer Verlag",
    doi = "10.1007/978-3-642-21919-1"
}

@misc{repo,
  author = {Jakob Führer and Georg Grasegger and Paul Hametner and Oliver Roche-Newton},
  title = {Code for the minimum degree question for the Maker Breaker Domination Game},
  howpublished = {https://github.com/jakobfuehrer/MBDChecker},
  year = {2026}
}

\end{document}